\documentclass[12pt]{amsart}
\usepackage{amssymb}
\usepackage[all]{xy}

\usepackage{verbatim}
\usepackage{eucal}
\usepackage{latexsym}
\usepackage{amsfonts,amssymb,amsxtra}
\usepackage{relsize}
\usepackage[dvips]{color}
\usepackage{latexsym}
\usepackage{mathptmx}
\usepackage{amssymb}
\usepackage{graphicx}
\usepackage{amsthm}
\usepackage{amsfonts}
\usepackage{amscd}
\usepackage{bm}
\usepackage[utf8]{inputenc}

\usepackage{verbatim} 
\usepackage{eucal} 
\usepackage{color}

\usepackage[draft]{todonotes}
\usepackage{pgf}
\usepackage{tikz}
\usetikzlibrary{arrows,automata}

\usepackage{pgf, tikz}
\definecolor {processblue}{cmyk}{0.96,0,0,0}
\input xy
\xyoption{all}
\usepackage{mathtools}
\usepackage{amsfonts}
\usepackage{amsthm}
\usepackage{mathrsfs}

\usepackage{amscd}
  \newtheorem{The}{Theorem}[section]
  \newtheorem{Pro}[The]{Proposition}
  \newtheorem{Lem}[The]{Lemma}
  \newtheorem{Cor}[The]{Corollary}

  \newtheorem{Examp}[The]{Example}

\newcommand{\bsm}{\begin{smallmatrix}}
\newcommand{\esm}{\end{smallmatrix}}
\newcommand{\bbm}{\begin{matrix}}
\newcommand{\ebm}{\end{matrix}}

\newcommand{\Hom}{\rm{Hom}}

\theoremstyle{definition}

\theoremstyle{plain}

\theoremstyle{definition}

\numberwithin{equation}{section}

\begin{document}

\title[Locally finitely presented Grothendieck categories]{Locally finitely presented Grothendieck categories and the pure semisimplicity conjecture}
\newcommand\shortTitle{Locally finitely presented Grothendieck categories}
\author{Ziba Fazelpour}
\address{School of Mathematics, Institute for Research in Fundamental Sciences (IPM), P.O. Box: 19395-5746, Tehran, Iran}
\email{z.fazelpour@ipm.ir}
\author{Alireza Nasr-Isfahani}
\address{Department of Pure Mathematics\\
Faculty of Mathematics and Statistics\\
University of Isfahan\\
P.O. Box: 81746-73441, Isfahan, Iran\\ and School of Mathematics, Institute for Research in Fundamental Sciences (IPM), P.O. Box: 19395-5746, Tehran, Iran}
\email{nasr$_{-}$a@sci.ui.ac.ir / nasr@ipm.ir}



\maketitle
{\centering\footnotesize Dedicated to the memory of Daniel Simson\par}

\begin{abstract}
In this paper, we investigate locally finitely presented pure semisimple (hereditary) Grothendieck categories. We show that every locally finitely presented pure semisimple (resp., hereditary) Grothendieck category $\mathscr{A}$ is equivalent to the category of left modules over a left pure semisimple (resp., left hereditary) ring when ${\rm Mod}({\rm fp}(\mathscr{A}))$ is a QF-3 category and every representable functor in ${\rm Mod}({\rm fp}(\mathscr{A}))$ has finitely generated essential socle. In fact, we show that there exists a bijection between Morita equivalence classes of left pure semisimple (resp., left hereditary) rings $\Lambda$ and equivalence classes of locally finitely presented pure semisimple (resp., hereditary) Grothendieck categories $\mathscr{A}$ that ${\rm Mod}({\rm fp}(\mathscr{A}))$ is a QF-3 category and every representable functor in ${\rm Mod}({\rm fp}(\mathscr{A}))$ has finitely generated essential socle. To prove this result, we study left pure semisimple rings by using Auslander’s ideas. We show that there exists, up to equivalence, a bijection between the class of left pure semisimple rings and the class of rings with nice homological properties. These results extend the Auslander and Ringel-Tachikawa correspondence to the class of left pure semisimple rings. As a consequence, we give several equivalent statements to the pure semisimplicity
conjecture.
\end{abstract}

\section{Introduction}
A cocomplete abelian category $\mathscr{A}$ is called {\it Grothendieck} if direct limits are exact in $\mathscr{A}$ and $\mathscr{A}$ has a generator (see \cite{gro}). The category ${\rm Qcoh}(X)$ of quasi-coherent sheaves over any quasi-compact and quasi-separated algebraic scheme and the quotient category $\mathscr{G}/\mathscr{T}$, where $\mathscr{G}$ is a locally finitely presented Grothendieck category and $\mathscr{T}$ is a hereditary torsion class in $\mathscr{G}$ generated by finitely presented objects  are examples of locally finitely presented Grothendieck categories (see \cite[I.6.9.12]{gra} and \cite[Proposition 2.4]{es}).

A Grothendieck category is called {\it pure semisimple} if each of its objects is a coproduct of finitely presented objects. The category of representations of any quiver of the pure semisimple type which was studied by Drozdowski and Simson \cite{ds} is one of the important examples of locally finitely presented pure semisimple Grothendieck categories. The pure semisimplicity of the category Comod-$C$ of right $C$-comodules over a coalgebra $C$ studied by using the right Gabriel quiver $Q_C$ of $C$ in \cite[Sect. 7]{S77}, \cite[p. 404]{s82}, and \cite{ko, no, s01, s04}. Note that the category Comod-$C$ of right $C$-comodules over a coalgebra $C$ is also one of the examples of locally finitely presented Grothendieck categories.

In \cite{S80}, Simson showed that there exists a bijection between equivalence classes of pure semisimple (resp., hereditary) Grothendieck categories and equivalence classes of  pure semisimple (resp., hereditary) functor categories Mod$(\mathscr{C}^{\rm op})$, where $\mathscr{C}$ has pseudocokernels. In \cite{gs1}, Garcia and Simson studied locally finitely presented pure semisimple Grothendieck locally PI-categories. They showed that any locally finitely presented pure semisimple Grothendieck locally PI-category $\mathscr{A}$ is equivalent to the category of ${\rm Mod}\widetilde{\Lambda}$ of unitary right $\widetilde{\Lambda}$-modules over a right perfect ring  $\widetilde{\Lambda}$ (in general $\widetilde{\Lambda}$ has no identity) with a complete set of orthogonal primitive local idempotents when $\mathscr{A}$ is equivalent to the category of unitary modules over a ring with enough idempotents (see \cite[Sect. 3.3, Theorem 2]{gs1}). Moreover, they proved that any locally finitely presented pure semisimple Grothendieck locally PI-category $\mathscr{A}$ is equivalent to the category of unitary right modules over an artinian PI-ring of finite representation type $\Lambda$ when $\mathscr{A}$ has only finitely many non-isomorphic simple objects (see \cite[Sect. 3.3, Corollary 7]{gs1}). A left artinian ring $\Lambda$ is called {\it of finite representation type} if it has, up to isomorphism, only finitely many finitely generated indecomposable left $\Lambda$-modules. Note that there exist locally finitely presented pure semisimple Grothendieck categories which are not of finite representation type (see \cite{ds}). Therefore, the class of locally finitely presented pure semisimple Grothendieck categories is bigger than the class of the module categories over pure semisimple rings. 

A ring $\Lambda$ is called {\it left pure semisimple} if the category $\Lambda$-Mod is pure semisimple. Equivalently, every left $\Lambda$-module is a direct sum of finitely generated left $\Lambda$-modules. Fuller and Reiten \cite{fu, fr} proved that a ring $\Lambda$ is left and right pure semisimple if and only if $\Lambda$ is of finite representation type. The problem of whether left pure semisimple rings are of finite representation type, known as the pure semisimplicity conjecture, remains open (see \cite{au, s77, s80, z}). Auslander in \cite{au} showed that the pure semisimplicity conjecture is valid for Artin algebras. Also, Herzog in \cite{her} showed that if there exists a counterexample to the pure semisimplicity conjecture,
then there exists a counterexample $R$ which is hereditary.\\

In 1997, Garcia and Simson \cite{gs}, studied locally finitely presented Grothendieck categories by using their (Gabriel) functor rings. They showed that there exists, up to equivalence, a bijection between the class of all locally finitely presented Grothendieck categories and the class of all left panoramic rings with enough idempotents. In this bijection, the category corresponding to a given ring $R$ is the category of all flat unitary left $R$-modules. Also, the ring corresponding to a given category $\mathscr{A}$ is  the functor ring $T_{\mathscr{A}}$ of $\mathscr{A}$. The functor rings are an important tool in the study of locally finitely presented pure semisimple Grothendieck categories. As an example, in \cite{S77}, it was shown that a locally finitely presented Grothendieck category $\mathscr{A}$ is pure semisimple if and only if $T_{\mathscr{A}}$ is a left perfect ring. In this paper, we study locally finitely presented pure semisimple (hereditary) Grothendieck categories by using this technique. We show that any locally finitely presented pure semisimple (resp., hereditary) Grothendieck category which satisfies the following property
\begin{itemize}
\item[$(\ast)$] ${\rm Mod}({\rm fp}(\mathscr{A}))$ is a QF-3 category and every representable functor in ${\rm Mod}({\rm fp}(\mathscr{A}))$ has finitely generated essential socle.
\end{itemize}
is equivalent to the category of left modules over a left pure semisimple (resp., left hereditary) ring. Then we provide a bijection between these two class of categories up to equivalence. To obtain this result, we first study left pure semisimple rings by using Auslander’s ideas. Auslander in his famous theorem \cite{Ausla2}, which is called Auslander correspondence, provided a bijection between Morita equivalence classes of rings of finite representation type and Morita equivalence classes of Auslander rings. A left artinian ring $R$ is called {\it Auslander ring} if the global dimension of $R$ is less than or equal 2 and the dominant dimension of $R$ is greater than or equal 2. This result plays a crucial role in representation theory of artin algebras. Ringel and Tachikawa in \cite{rt} by using QF-3 maximal quotient rings gave a short proof for the Auslander's result.
It is natural to ask about similar bijection for the class of left pure semisimple rings. In this paper we extend the Auslander-Ringel-Tachikawa correspondence to the class of left pure semisimple rings. Note that the functor rings that appeared in the proof of the Auslander correspondence are unital, but in our case the functor rings are not unital (these are rings with enough idempotents). 

Consider the following properties for a ring $R$ with enough idempotents.
\begin{itemize}
\item[$(i)$] There is a complete set $\lbrace e_{\alpha}~|~\alpha \in J \rbrace$ of pairwise orthogonal local idempotents of $R$ such that $R=\bigoplus_{\alpha \in J}Re_{\alpha}=\bigoplus_{\alpha \in J}e_{\alpha}R$ and each $e_{\alpha}Re_{\alpha}$ is a right perfect ring;
\item[$(ii)$] $R$ is a left locally noetherian ring with ${\rm l.gl.dim}R = 0 ~{\rm or} ~ 2$;
\item[$(iii)$] ${\rm proj}(R)\bigcap {\rm inj}(R)= {\rm add}(\bigoplus_{l=1}^n E(S_l))$, where $\lbrace S_1,\cdots,S_n\rbrace$ is a set of non-isomorphic simple unitary left $R$-modules;
\item[$(iv)$] For each $\alpha \in J$, {\rm dom.dim} $Re_{\alpha} \geq 2$.\\
\end{itemize}

As a first main result of this paper, we prove the following theorem.\\

\textbf{Theorem A.} (Theorem \ref{re1}) {\it There exists a bijection between Morita equivalence classes of left pure semisimple rings $\Lambda$ and Morita equivalence classes of rings $R$ with the properties $(i)$-$(iv)$.}\\

It is easy to see that any Auslander ring has the properties $(i)$-$(iv)$. It is natural to ask about the converse of this fact. We show that the converse of this fact is equivalent to the pure semisimplicity conjecture (see Corollary \ref{re8}).

Then, we prove the following result which extends the Ringel-Tachikawa correspondence to the class of left pure semisimple rings.\\

\textbf{Theorem B.} (Theorem \ref{rt1}) {\it There exists a bijection between Morita equivalence classes of left pure semisimple rings $\Lambda$ and Morita equivalence classes of left perfect right locally coherent QF-3 rings $R$ with enough idempotents which have a minimal faithful balanced right ideal and ${\rm l.gl.dim}R = 0 ~{\rm or} ~ 2$.}\\

We show that there exists, up to equivalence, a bijection between the class of all locally finitely presented pure semisimple Grothendieck categories which satisfy the property $(\ast)$ and the class of rings with the properties $(i)$-$(iv)$ (see Theorem \ref{r53}). In this bijection, the category corresponding to a given ring $R$ is the category of all flat unitary left $R$-modules. Also, the ring corresponding to a given category $\mathscr{A}$ is the functor ring $R_{\mathscr{A}}$ of $\mathscr{A}$. Next, by using Theorem A, we prove the following result.\\

\textbf{Theorem C.} (Corollary \ref{r54}) {\it There exists a bijection between Morita equivalence classes of left pure semisimple (resp., left hereditary) rings $\Lambda$ and equivalence classes of locally finitely presented pure semisimple (resp., hereditary) Grothendieck categories which satisfy the property $(\ast)$.}\\

As a consequence of above theorem and \cite[Theorem 6.9]{her}, we have the following corollary which gives an equivalent statement to the pure semisimplicity
conjecture.\\

\textbf{Corollary A.} (Corollary \ref{p22}) {\it The following statement are equivalent.
\begin{itemize}
\item[$(a)$] The pure semisimplicity conjecture is true;
\item[$(b)$] Every locally finitely presented pure semisimple hereditary Grothendieck category which satisfies the property $(\ast)$ is of finite representation type.\\
\end{itemize}}

We hope these results create a new path to prove or disprove of the pure semisimplicity conjecture.\\

The paper is organized as follows. In Section 2, we prove some preliminary results that will be needed throughout the paper. In Sections 3 and 4, we extend and generalize the Auslander correspondence and Ringel-Tachikawa correspondence to the class of left pure semisimple rings and prove Theorems A and B, respectively. In Section 5, we study locally finitely presented pure semisimple hereditary Grothendieck categories which satisfy the property $(\ast)$ and prove Theorem C.

\subsection{Notation }
In this paper, we write morphisms in a category $\mathscr{C}$ on the right side of the objects. Let $\mathscr{C}$ be an additive category.  We denote by Mod$(\mathscr{C})$ the category of all contravariant additive functors from $\mathscr{C}$ to the category $\mathfrak{Ab}$ of all abelian groups. We also denote by ${\rm Flat}(\mathscr{C})$ (resp., ${\rm proj}(\mathscr{C})$) the full subcategory of Mod$(\mathscr{C})$ whose objects are flat (resp., finitely generated projective) functors; recall that a functor $F$ in Mod$(\mathscr{C})$ is called flat if $F$ is a direct limit of finitely generated projective functors. Let $F: \mathscr{C} \rightarrow \mathscr{C'}$ be an additive functor and let $\mathscr{D}$ be a full subcategory of $\mathscr{C}$. We denote by $F|_{\mathscr{D}}$ the restriction of $F$ to $\mathscr{D}$. Throughout this paper all rings are associative, not necessary with unit. Let $R$ be a ring. We denote by $R$-Mod (resp., Mod-$R$) the category of all left (resp., right) $R$-modules. Also we denote by $R$-mod (resp., mod-$R$) the category of all finitely generated left (resp., right) $R$-modules and by $J(R)$ the Jacobson radical of $R$. A left $R$-module $M$ is called {\it unitary} if $RM=M$. We denote by $R$Mod (resp., Mod$R$) the category of all unitary left (resp., right) $R$-modules. We write $\prod_A M$ for the direct product of cardinal $A$ copies of an $R$-module $M$. Moreover, we write $\widehat{\prod}_B M=R \prod_B M$ for the direct product of cardinal $B$ copies of a unitary left $R$-module $M$ in $R$Mod. Let $N$ be a unitary left $R$-module. We denote by $E(N)$, ${\rm rad}(N)$, ${\rm top}(N)$ and ${\rm Soc}(N)$ the injective hull of $N$ in $R$Mod, radical of $N$, top of $N$ and socle of $N$, respectively. For a left $R$-module $V$, we denoted by ${\rm Add}(V)$ (resp., ${\rm add}(V)$) the full subcategory of $R$Mod whose objects are all left $R$-modules that are isomorphic to direct summands of direct sum (resp., finite direct sum) of copies of $V$. We denote by ${\rm Proj}(R)$ (resp., ${\rm proj}(R)$) the full subcategory of $R$Mod consisting of (resp., finitely generated) unitary left $R$-modules which are projective in $R$Mod. Also we denote by ${\rm inj}R$ the full subcategory of $R$Mod consisting of finitely generated unitary left $R$-modules which are injective in $R$Mod. Throughout this paper, we assume that $\Lambda$ is a ring with identity.

\section{Preliminaries}

A ring $R$ is called a {\it ring with local units} if every finite subset of $R$ contained in a subring of the form $eRe$ where $e^2 = e \in R$ (see \cite[p. 464]{wi}). Throughout this section we assume that all rings are associative with local units.\\

Let $(I, \leq)$ be a quasi-ordered directed set. A direct system of unitary left $R$-modules\\ ${(X_i, \varphi_{ij})_{i,j \in I}}$ is called {\it split} if for each $i \leq j$ in $I$, there exists an $R$-module homomorphism $\psi_{ji}: X_j \rightarrow X_i$ such that $\varphi_{ij}\psi_{ji}=id_{X_i}$ and $\psi_{kj}\psi_{ji}=\psi_{ki}$ for each $i \leq j \leq k$ in $I$ (see \cite[p. 6]{zn}). A unitary left $R$-module $P$ is called {\it locally projective} if $P$ is a direct limit of a split direct system $(P_i)_{i \in I}$, where each $P_i$ is a finitely generated  projective unitary left $R$-module (see \cite[p. 6]{zn}).\\

Let $R$ and $S$ be two rings and $P$ be an $R$-$S$-bimodule.  Assume that $P$ is a locally projective unitary left $R$-module such that $Pf$ is a finitely generated left $R$-module for each idempotent $f \in S$. By \cite[Lemma 2.5]{zn}, the functor $R{\rm Hom}_S(G_P,-): S{\rm Mod}\rightarrow R{\rm Mod}$ is a right adjoint of the functor $S{\rm Hom}_R(P,-): R{\rm Mod} \rightarrow S{\rm Mod}$, where $G_P=S{\rm Hom}_R(P,R)$. For each unitary left $R$-module $M$, we have an $R$-module homomorphism  $\gamma_M: M \rightarrow R{\rm Hom}_S(G_P,S{\rm Hom}_R(P,M))$ defined by $(\beta)(m)\gamma_M=(\beta \otimes m)\eta_M:P \rightarrow M$ via $x \mapsto (x)\beta m$ for each $m \in M$ and $\beta \in G_P$. Therefor we have a natural transformation $\gamma: id_{R{\rm Mod}} \rightarrow R{\rm Hom}_S(G_P, S{\rm Hom}_R(P,-))$, where $\gamma=(\gamma_M)$ (see \cite[Proposition 45.5]{wi}).\\

In the following proposition and corollary, we assume that $P$ is a locally projective left $R$-module, $S$ is a subring of ${\rm End}_R(P)$ such that $P$ is a unitary right $S$-module, $S{\rm End}_R(P)=S$ and $Pf$ is a finitely generated left $R$-module for each $f^2=f \in S$. Note that the following proposition and corollary are non-unital analogs of \cite[Propositions 5.5 and 5.6]{Ausla1}.

\begin{Pro}\label{b2}
The following conditions are equivalent for any unitary left $R$-module $N$.
\begin{itemize}
\item[$(a)$] The canonical morphism $\gamma_N: N \rightarrow R{\rm Hom}_S(G_P,S{\rm Hom}_R(P,N))$ is an isomorphism.
\item[$(b)$] There exists a unitary left $S$-module $N'$ such that $R{\rm Hom}_S(G_P,N')\cong N$ as $R$-modules.
\item[$(c)$] There exists an exact sequence of $R$-modules $$0\rightarrow N \rightarrow R{\rm Hom}_S(G_P,I_0) \rightarrow R{\rm Hom}_S(G_P,I_1),$$ where $I_0$ and $I_1$ are injective in $S${\rm Mod}.
\item[$(d)$] For each $M \in R{\rm Mod}$, the functor $S{\rm Hom}_R(P,-)$ gives an isomorphism of abelian groups $${\rm Hom}_R(M,N) \rightarrow {\rm Hom}_S(S{\rm Hom}_R(P,M),S{\rm Hom}_R(P,N)).$$
\end{itemize}
\end{Pro}
\begin{proof}
$(a) \Rightarrow (b)$ is clear.\\
$(b) \Rightarrow (d)$. Consider the exact sequence $0\rightarrow N' \rightarrow I_0 \overset{t}{\rightarrow} I_1$, where $I_0$ and $I_1$ are injective modules in $S$Mod.  Then the sequence $0\rightarrow N \overset{h}{\rightarrow} R{\rm Hom}_S(G_P,I_0) \overset{g}{\rightarrow} R{\rm Hom}_S(G_P,I_1)$ is exact, where $g=R{\rm Hom}_S(G_P,t)$. This implies that the following sequence is also exact $$0\longrightarrow S{\rm Hom}_R(P,N) \overset{S{\rm Hom}_R(P,h)}{\longrightarrow} S{\rm Hom}_R(P,R{\rm Hom}_S(G_P,I_0)) \overset{S{\rm Hom}_R(P,g)}{\longrightarrow} S{\rm Hom}_R(P,R{\rm Hom}_S(G_P,I_1)).$$
We know that by \cite[Lemma 2.1]{zn}, the pair of functors
\begin{center}
$G_P\otimes_R-:R{\rm Mod} \rightarrow S{\rm Mod}$ \hspace{2mm} and \hspace{2mm} $R{\rm Hom}_S(G_P,-):S{\rm Mod} \rightarrow R{\rm Mod}$
\end{center}
is adjoint via the isomorphisms {\rm(}for $X \in R${\rm Mod} and $Y \in S${\rm Mod)}
\begin{center}
$\psi_{X,Y}:{\rm Hom}_R(X,R{\rm Hom}_S(G_P,Y)) \rightarrow {\rm Hom}_S(G_P\otimes_RX,Y)$ via $\beta \mapsto \left[ \sum_{i=1}^t\alpha_i\otimes x_i\mapsto \sum_{i=1}^t(\alpha_i)(x_i)\beta \right].$
\end{center}
Also by \cite[Lemma 2.7]{zn}, for each $Y \in S${\rm Mod} there exists an $S$-module isomorphism $$\nu_Y:Y \rightarrow S{\rm Hom}_R(P,R{\rm Hom}_S(G_P,Y))$$
which determines a natural isomorphism ${\rm id}_{S{\rm Mod}} \rightarrow S{\rm Hom}_R(P,R{\rm Hom}_S(G_P,-))$. Moreover by the proof of \cite[Lemma 2.5]{zn}, there exists an $S$-module isomorphism $$\eta_M: G_P \otimes_R M  \rightarrow S{\rm Hom}_R(P,M).$$ Set $I'_i:=S{\rm Hom}_R(P,R{\rm Hom}_S(G_P,I_i))$ and $\delta_i:= {\rm Hom}_R(M,R{\rm Hom}_S(G_P,\nu_{I_i}))$, for $i=0,1$ and let $M$ be a unitary left $R$-module. Therefore the following diagram is commutative
\begin{displaymath}
\xymatrix{
0 \ar[r] & {\rm Hom}_R(M,N)\ar@{->}^<<<<{{\rm Hom}_R(M,h)}[r]  &
{\rm Hom}_R(M,R{\rm Hom}_S(G_P,I_0)) \ar[r]^{{\rm Hom}_R(M,g)} \ar[d]^{\delta_0} &{\rm Hom}_R(M,R{\rm Hom}_S(G_P,I_1)) \ar[d]^{\delta_1} \\
&& {\rm Hom}_R(M,R{\rm Hom}_S(G_P,I'_0))  \ar[r]^{g'''} \ar[d]^{\Psi_{M,I'_0}} &{\rm Hom}_R(M,R{\rm Hom}_S(G_P,I'_1)) \ar[d]^{\Psi_{M,I'_1}}\\
&&  {\rm Hom}_S(G_P\otimes_R M,I'_0) \ar[r]^{{\rm Hom}_S(G_P\otimes_R M,g')} \ar[d]^{{\rm Hom}_S(\eta^{-1}_M,I'_0)} &{\rm Hom}_S(G_P\otimes_R M,I'_1)\ar[d]^{{\rm Hom}_S(\eta^{-1}_M,I'_1)}\\
0 \ar[r] &{\rm Hom}_S(M',N') \ar[r]_<<<<<<<{h''} & {\rm Hom}_S(M',I'_0) \ar[r]_<<<<<<<{g''} &{\rm Hom}_S(M',I'_1) }
\end{displaymath}
where $M'=S{\rm Hom}_R(P,M)$, $N'=S{\rm Hom}_R(P,N)$, $g'=S{\rm Hom}_R(P,g)$, $h'=S{\rm Hom}_R(P,h)$, $g'''={\rm Hom}_R(M,R{\rm Hom}_S(G_P,g'))$, $h''={\rm Hom}_S(S{\rm Hom}_R(P,M),h')$ and $g''={\rm Hom}_S(S{\rm Hom}_R(P,M),g')$. By using the proof of \cite[Lemmas 2.5 and 2.7]{zn} and the five lemma we can see that the functor $S{\rm Hom}_R(P,-)$ gives the following isomorphism $${\rm Hom}_R(M,N) \rightarrow {\rm Hom}_S(S{\rm Hom}_R(P,M),S{\rm Hom}_R(P,N)).$$
$(d) \Rightarrow (a).$ Put $M=R$ and then we can easily see this isomorphism induces asked isomorphism. \\
$(b) \Rightarrow (c)$ is clear.\\
$(c) \Rightarrow (a).$ It follows from $(b) \Leftrightarrow (a)$ and the five lemma.

\end{proof}

A unitary left $R$-module $M$ is called {\it injectively copresented over} $P$ if the canonical $R$-module homomorphism $\gamma_M: M \rightarrow R{\rm Hom}_S(G_P,S{\rm Hom}_R(P,M))$ is an isomorphism. We denote the full subcategory of $R$Mod consisting of all left $R$-modules which are injectively copresented over $P$ by $R{\rm Hom}_S(G_P, S{\rm Mod})$ (see \cite[p. 221]{Ausla1}). Notice that the category $R{\rm Hom}_S(G_P, S{\rm Mod})$ is closed under direct summands.\\

As a consequence of \cite[Lemma 2.7]{zn} and Proposition \ref{b2} we have the following corollary.

\begin{Cor}\label{b3}
The functor $S{\rm Hom}_R(P,-): R{\rm Hom}_S(G_P, S{\rm Mod}) \rightarrow S{\rm Mod}$ is an equivalence of categories.
\end{Cor}

A ring $R$ is called {\it semiperfect} if every finitely generated unitary left $R$-module has a projective cover in $R$Mod (see \cite[p. 334]{ha1}). Let $R$ and $S$ be two rings, $M$ be a unitary left $R$-module and $P$ be an $R$-$S$-bimodule. We denoted by ${\rm st}_P(M)$ the sum of all $R$-submodules $K$ of $M$ with $S{\rm Hom}_R(P,K)=0$ (see \cite[p. 5]{zn}). We need a non-unital version of Proposition 5.7(d) and Corollary 7.4(c) of \cite{Ausla1} for non-unital semiperfect (non-artinian) rings.

\begin{Lem}\label{c1}
Let $R$ be a semiperfect ring and let $\lbrace S_1,\cdots,S_n\rbrace$ be a finite set of non-isomorphic simple unitary left $R$-modules.  Assume that for each $i$, $f_i: P_i \rightarrow S_i$ is a projective cover of $S_i$ in $R${\rm Mod}, $P=\bigoplus_{i=1}^n P_i$ and $S={\rm End}_R(P)$. If $E$ is an injective module in $R${\rm Mod} such that ${\rm Soc}(E)$ is an essential submodule of $E$, then the following conditions are equivalent.
\begin{itemize}
\item[$(a)$] ${\rm Soc}(E) \in {\rm Add}(\bigoplus_{i=1}^nS_i)$.
\item[$(b)$] $st_P(E)=0$.
\item[$(c)$] $E \in R{\rm Hom}_S(G_P, S{\rm Mod})$.
\end{itemize}
\end{Lem}
\begin{proof}
$(a) \Rightarrow (b)$. Since ${\rm Soc}(E) \in {\rm Add}(\bigoplus_{i=1}^nS_i)$,  $st_P({\rm Soc}(E))=0$. Then $st_P(E)=0$ because ${\rm Soc}(E)$ is an essential submodule of $E$ and $st_P(E) \cap {\rm Soc}(E) \subseteq st_P({\rm Soc}(E))$. \\
$(b)\Rightarrow (a)$. Assume that there exists a simple direct summand $S'$ of ${\rm Soc}(E)$ such that $S' \ncong S_i$ for each $i~(1 \leq i \leq n)$. Then $st_P(S') \neq 0$ since otherwise there is a non-zero morphism $g \in {\rm Hom}_R(P,S')$. Then there exists a non-zero $R$-module homomorphism $h: P \rightarrow P(S')$ such that $hf=g$, where $f:P(S') \rightarrow S'$ is a projective cover of $S'$ in $R$Mod. Hence $S' \cong S_j$ for some $j$ which is a contradiction. Therefore $st_P(S') \neq 0$ and so $st_P(E) \neq 0$ which is a contradiction. Consequence ${\rm Soc}(E) \in {\rm Add}(\bigoplus_{i=1}^nS_i)$. \\
$(b)\Rightarrow (c)$. It follows from \cite[Lemma 2.8]{zn}.\\
$(c)\Rightarrow (b)$. From $E$ is injectively copresented over $P$, by \cite[Lemma 2.5]{zn} we get $${\rm Hom}_R(st_P(E),E) \cong {\rm Hom}_S({\rm Hom}_R(P,st_P(E)),{\rm Hom}_R(P,E))$$  as abelian groups. Since ${\rm Hom}_R(P,st_P(E))=0$, ${\rm Hom}_R(st_P(E),E)=0$ and so  $st_P(E)=0$.
\end{proof}

\section{Gabriel rings and left pure semisimple rings}
In this section we study the left functor (Gabriel) rings of left pure semisimple rings. We show that the left functor ring of any left pure semisimple ring has the properties $(i)$-$(iv)$ (see Corollary \ref{cc1}). Then we show that the properties $(i)$-$(iv)$ classify left pure semisimple rings up to Morita equivalence (see Theorem \ref{re1}). As a consequences we show that every ring with the properties $(i)$-$(iv)$ is Morita equivalent to an Auslander ring if and only if the pure semisimplicity conjecture holds (see Corollary \ref{re8}). \\

A family $\lbrace U_i~|~i \in I\rbrace$ of objects of an abelian category $\mathscr{C}$ is a family of {\it generators} (resp., {\it cogenerators}) for $\mathscr{C}$ if for each non-zero morphism $f: M \rightarrow N$ in $\mathscr{C}$ there exists a morphism $h : U_i \rightarrow M$ (resp., $h : N \rightarrow U_i$), such that $hf \neq 0$ (resp., $fh \neq 0$) (see \cite[p. 94]{bs}). A Grothendieck category $\mathscr{C}$ is said to be {\it locally noetherian} if it has a generating set of noetherian objects (see \cite[p. 123]{bs}). We recall that a ring $R$ has {\it enough idempotents} if there exists a family $\lbrace q_{\alpha}~|~\alpha \in I \rbrace$ of pairwise orthogonal idempotents of $R$ such that $R=\bigoplus_{\alpha \in I}Rq_{\alpha}=\bigoplus_{\alpha \in I}q_{\alpha}R$ (see \cite[p. 39]{fu}). Let $R$ be a ring with enough idempotents. $R$ is called {\it left locally noetherian} if the category $R$Mod is locally noetherian. We recall that {\it dominant dimension} of a unitary left $R$-module $X$ is greater than or equal to $2$, denoted by dom.dim $X \geq 2$, if there is a minimal injective resolution $0 \rightarrow X \rightarrow E_1 \rightarrow E_2$ such that each $E_i$ ($i=1,2$) is a projective unitary left $R$-module. We denoted by ${\rm l.dom.dim}R$ the dominant dimension of left $R$-module $R$. The {\it left global dimension} of $R$, which is denoted by ${\rm l.gl.dim}(R)$, is the supremum of the set of projective dimensions of all unitary left $R$-module.\\

Let $\mathscr{A}$ be an additive category with direct limits. An object $A$ in $\mathscr{A}$ is called {\it finitely presented} if the representable functor $\mathscr{A}(A,-)={\rm Hom}_{\mathscr{A}}(A,-): \mathscr{A} \rightarrow \mathfrak{Ab}$ preserves direct limits (see \cite[p. 1642]{wc}). We denote by ${\rm fp}(\mathscr{A})$ the full subcategory of finitely presented objects in $\mathscr{A}$. We recall from \cite[p. 1642]{wc} that $\mathscr{A}$ is called {\it locally finitely presented} if ${\rm fp}(\mathscr{A})$ is skeletally small and $\mathscr{A}={\underrightarrow{\lim}}~{\rm fp}(\mathscr{A})$ (i.e., every object in $\mathscr{A}$ is a direct limit of finitely presented objects in $\mathscr{A}$). Note that every locally finitely presented Grothendieck category $\mathscr{A}$ has a generating set of finitely presented objects. An object $C$ of a Grothendieck category $\mathscr{C}$ is called {\it finitely generated} whenever $C=\sum_{i}C_i$ for a direct family of subobjects $C_i$ of $C$, there exists an index $i_0$ such that $C=C_{i_0}$ (see \cite[p. 121]{bs}). Every finitely presented object in a locally finitely presented Grothendieck category $\mathscr{A}$ is finitely generated. By \cite[Ch.  V, Proposition 3.4]{bs}, an object $X$ in a locally finitely presented Grothendieck category $\mathscr{A}$ is finitely presented if and only if $X$ is finitely generated and every epimorphism $Y \rightarrow X$, where $Y$ is finitely generated, has a finitely generated kernel.\\

Let $\mathscr{A}$ be a locally finitely presented Grothendieck category and $\lbrace U_{\alpha}~|~ \alpha \in J \rbrace$ be a family of finitely presented objects in $\mathscr{A}$. Set $U=\bigoplus_{\alpha \in J} U_{\alpha}$ and for each $\alpha \in J$ let $e_{\alpha}=\pi_{\alpha} \varepsilon_{\alpha}$, where $\pi_{\alpha}:U \rightarrow U_{\alpha}$ is the canonical projection and $\varepsilon_{\alpha}:U_{\alpha} \rightarrow U$ is the canonical injection. For each object $X$ of $\mathscr{A}$, we define as in \cite[p. 296]{aw} (see also, \cite[p. 40]{fu}), $\widehat{{\Hom}}_{\mathscr{A}}(U,X)=\lbrace f \in {\rm Hom}_{\mathscr{A}}(U,X)~|~ e_{\alpha}f =0~ {\rm for ~almost~ all}~ \alpha \in J \rbrace$.  For $X=U$, we write $R:= \widehat{{\rm Hom}}_{\mathscr{A}}(U, U)=\widehat{{\rm End}}_{\mathscr{A}}(U)$. Then $R=\bigoplus_{\alpha \in J}Re_{\alpha}=\bigoplus_{\alpha \in J}e_{\alpha}R$ is a ring with enough idempotents. Fuller in \cite[p. 40]{fu} (see also \cite[Lemma 1.1]{aw}) defined a covariant functor $\widehat{{\Hom}}_{\mathscr{A}}(U,-): \mathscr{A} \rightarrow R$Mod as follows. For any morphism $f: X \rightarrow Y$ in $\mathscr{A}$, he defined $\widehat{{\rm Hom}}_{\mathscr{A}}(U,f): \widehat{{\rm Hom}}_{\mathscr{A}}(U,X) \rightarrow \widehat{{\rm Hom}}_{\mathscr{A}}(U,Y)$ via $g \mapsto gf$. From \cite[p. 40-41]{fu} (see also \cite[Lemma 1.1]{aw}) we observe that the left exact covariant functor $\widehat{{\Hom}}_{\mathscr{A}}(U,-)$  preserves coproducts and products when $U$ is a generator in $\mathscr{A}$.\\

We recall that a set $\lbrace e_1, \cdots, e_m\rbrace$ of idempotents of a semiperfect ring $\Lambda$ is called {\it basic} in case they are pairwise orthogonal, $\Lambda e_i\ncong \Lambda e_j$ for each $i\neq j$  and for each finitely generated indecomposable projective left $\Lambda$-module $P$, there exists $i$ such that $P \cong \Lambda e_i$. Clearly, the cardinal of any two basic sets of idempotents of a semiperfect ring $\Lambda$ are equal.\\

Two rings  with enough idempotents $R$ and $S$ are said to be {\it Morita equivalent} in case there exists an additive covariant equivalence between $R$Mod and $S$Mod (see \cite{wi}).

\begin{Pro}\label{r4}
Let $\Lambda$ be a left artinian ring and let $\lbrace V_{\alpha}~|~ \alpha \in J \rbrace$ be a family of finitely generated indecomposable left $\Lambda$-modules. Set $V=\bigoplus_{\alpha \in J} V_{\alpha}$ and $R=\widehat{{\rm End}}_{\Lambda}(V)$. Then $R$ has the property $(i)$. In particular,
\begin{itemize}
\item[$(a)$] If $V$ is a generator in $\Lambda$-{\rm Mod} and ${\rm l.gl.dim}R= 0 ~{\rm or} ~2$, then $R$ is a left locally noetherian ring.
\item[$(b)$] If $V$ is a generator-cogenerator in $\Lambda$-{\rm Mod}, then $R$ has the properties $(iii)$-$(iv)$. Moreover $S={\rm End}_R(\bigoplus_{i=1}^nP_i)$ is Morita equivalent to $\Lambda$, where each $P_i$ is a projective cover of the simple module $S_i$.
\end{itemize}
\end{Pro}
\begin{proof} It follows from \cite[Propositions 32.4 and 51.6]{wi} and the fact that $\widehat{{\Hom}}_{\Lambda}(V,V_{\alpha}) \cong Re_{\alpha}$ as $R$-modules, where $e_{\alpha}=\pi_{\alpha} \varepsilon_{\alpha}$,  $\pi_{\alpha}:U \rightarrow U_{\alpha}$ is the canonical projection and $\varepsilon_{\alpha}:U_{\alpha} \rightarrow U$ is the canonical injection.\\
$(a)$. It follows from \cite[Lemma 2.3]{fin}.\\
$(b)$. Since $V$ is a generator in $\Lambda$-Mod, by \cite[Proposition 51.7]{wi}, the functor $\widehat{{\Hom}}_{\Lambda}(V,-):\Lambda$-Mod$\rightarrow R$Mod preserves indecomposable modules, injective modules and essential extensions. Hence there exists a minimal injective resolution $0 \longrightarrow {\widehat{{\rm{Hom}}}}_{\Lambda}(V,V_{\alpha}) \longrightarrow {\widehat{{\rm{Hom}}}}_{\Lambda}(V,E(V_{\alpha})) \longrightarrow Y$, where  $Y$ is a direct summand of ${\widehat{{\rm{Hom}}}}_{\Lambda}(V,E(E(V_{\alpha})/V_{\alpha}))$. Let $\lbrace \Lambda e_1/J(\Lambda)e_1, \cdots ,\Lambda e_m/J(\Lambda)e_m \rbrace$ be a complete set of non-isomorphic simple left $\Lambda$-modules, where $\lbrace e_1, \cdots, e_m \rbrace $ is a basic set of idempotents of $\Lambda$. Since $\Lambda$ is a left artinian ring, the injective hull of any finitely generated left $\Lambda$-module belongs to be ${\rm add}(\bigoplus_{i=1}^mE(\Lambda e_i/J(\Lambda)e_i))$. On the other hand, since $V$ is a cogenerator in $\Lambda$-Mod, each $E(\Lambda e_l/J(\Lambda)e_l) \in {\rm add}(V)$. Since the functor $\widehat{{\Hom}}_{\Lambda}(V,-)$ preserves direct sums, ${\widehat{{\rm{Hom}}}}_{\Lambda}(V,E(V_{\alpha}))$ and $Y$ are finitely generated projective unitary left $R$-modules. Hence $R$ has the property $(iv)$.
\vspace{4mm}
Now we show that $R$ has the property $(iii)$.  Since $V$ is a generator, $\Lambda e_i \cong V_{\alpha_i}$ for some $\alpha_i \in J$. Let $\rho_i:\Lambda e_i \rightarrow \Lambda e_i/J(\Lambda)e_i$ be a projective cover of $\Lambda e_i/J(\Lambda)e_i$ and $\pi_{\alpha_i}: V \rightarrow V_{\alpha_i}$ be the canonical projection. Then $\pi_{\alpha_i} \rho_{\alpha_i} \rho_i$ is a non-zero morphism in $\widehat{{\rm{Hom}}}_{\Lambda}(V,\Lambda e_i/J(\Lambda)e_i)$, where $\rho_{\alpha_i}: V_{\alpha_i} \rightarrow \Lambda e_i$ is an isomorphism. Consider the $R$-module homomorphism $$\varphi:=\widehat{{\rm{Hom}}}_{\Lambda}(V,\pi_{\alpha_i} \rho_{\alpha_i} \rho_i): \widehat{{\rm{End}}}_{\Lambda}(V) \rightarrow \widehat{{\rm{Hom}}}_{\Lambda}(V,\Lambda e_i/J(\Lambda)e_i).$$  It is easy to see that $\varphi$ is a non-zero morphism. Also $J(R)\widehat{{\rm{End}}}_{\Lambda}(V) \subseteq {\rm Ker}\varphi$ since otherwise there exist $t \in J(R)$ and $\gamma \in \widehat{{\rm{End}}}_{\Lambda}(V)$ such that $t\gamma \pi_{\alpha_i} \rho_{\alpha_i} \rho_i \neq 0$. It follows that ${\rm Im}~t\gamma \pi_{\alpha_i} \rho_{\alpha_i} \nsubseteq {\rm Ker}\rho_i$. Since $\rho_i: \Lambda e_i \rightarrow \Lambda e_i/J(\Lambda)e_i$ is a projective cover of $\Lambda e_i/J(\Lambda)e_i$, the morphism $t\gamma \pi_{\alpha_i} \rho_{\alpha_i}: V \rightarrow \Lambda e_i$ is an epimorphism and so there is a $\Lambda$-module homomorphism $h: \Lambda e_i \rightarrow V$ such that $ht\gamma \pi_{\alpha_i} \rho_{\alpha_i} = id_{\Lambda e_i}$. Let $\varepsilon_{\alpha_i}: V_{\alpha_i} \rightarrow V$ be the canonical injection. Since $\pi_{\alpha_i} \rho_{\alpha_i} ht\gamma \pi_{\alpha_i} \rho_{\alpha_i} {\rho^{-1}_{\alpha_i}} \varepsilon_{\alpha_i}= \pi_{\alpha_i} \varepsilon_{\alpha_i}$, $\pi_{\alpha_i}\varepsilon_{\alpha_i} \in J(R)$. By \cite[Proposition 49.7]{wi}, this is a contradiction with the fact that $J(R)$ contains no idempotent. Consequently ${\rm Im}\varphi$ is a non-zero semisimple submodule of $\widehat{{\rm{Hom}}}_{\Lambda}(V,\Lambda e_i/J(\Lambda)e_i)$. It follows that ${\rm Soc}(\widehat{{\rm{Hom}}}_{\Lambda}(V,E(\Lambda e_i/J(\Lambda)e_i))) \neq 0$ and so\\ ${\widehat{{\rm{Hom}}}_{\Lambda}(V,E(\Lambda e_i/J(\Lambda)e_i))}$ has a simple essential socle. Hence $$\widehat{{\rm{Hom}}}_{\Lambda}(V,E(\Lambda e_i/J(\Lambda)e_i)) \cong E(R\pi_{\alpha_i} \rho_{\alpha_i} \rho_i)$$ as $R$-modules. Since $R\pi_{\alpha_i} \rho_{\alpha_i} \rho_i \cong R\pi_{\alpha_j} \rho_{\alpha_j} \rho_j$ if and only if $\Lambda e_i/J(\Lambda)e_i \cong \Lambda e_j/J(\Lambda)e_j$, so $\lbrace R\pi_{\alpha_i} \rho_{\alpha_i} \rho_i ~|~1 \leq i \leq m \rbrace$ is a set of non-isomorphic simple unitary left $R$-modules such that ${\rm add}(\bigoplus_{i=1}^m E(R\pi_{\alpha_i} \rho_{\alpha_i} \rho_i)) \subseteq {\rm proj}(R) \cap {\rm inj}(R)$.  But because for each finitely generated indecomposable projective unitary left $R$-module $Q$,  $Q \in {\rm add}(\bigoplus_{i=1}^m E(R\pi_{\alpha_i} \rho_{\alpha_i} \rho_i))$ when $Q \in {\rm inj}(R)$, then ${\rm proj}(R) \cap {\rm inj}(R) \subseteq {\rm add}(\bigoplus_{i=1}^m E(R\pi_{\alpha_i} \rho_{\alpha_i} \rho_i))$ and so $R$ has the property $(iii)$.
\vspace{4mm}
We know that the mapping $R \pi_{\alpha_i}\rho_{\alpha_i} \rightarrow R\pi_{\alpha_i}\rho_{\alpha_i} \rho_i$ defined by $r\pi_{\alpha_i}\rho_{\alpha_i} \longmapsto r\pi_{\alpha_i} \rho_{\alpha_i} \rho_i$ is a nonzero $R$-module epimorphism. Since $R\pi_{\alpha_i}\rho_{\alpha_i} = \widehat{{\rm{Hom}}}_{\Lambda}(V,\Lambda e_i)$, $R\pi_{\alpha_i}\rho_{\alpha_i}$ is a finitely generated indecomposable projective unitary left $R$-module. So it is a projective cover of $R\pi_{\alpha_i} \rho_{\alpha_i} \rho_i$. Set $e=e_1 + \cdots + e_m$. Since $P= \bigoplus_{i=1}^m\widehat{{\rm{Hom}}}_{\Lambda}(V,\Lambda e_i) \cong \widehat{{\rm{Hom}}}_{\Lambda}(V,\Lambda e)$ as $R$-modules, \begin{center}
$S= {\rm End}_R(P) \cong {\rm End}_R({\widehat{\rm{Hom}}}_{\Lambda}(V,\Lambda e)) \cong {\rm End}_{\Lambda}(\Lambda e) \cong e\Lambda e$
\end{center}
as rings. It follows that $S$ is Morita equivalent to $\Lambda$ and the result holds.
\end{proof}

Following \cite{S74, S77} (cf. \cite{wc}) a locally finitely presented Grothendieck category $\mathscr{A}$ is {\it pure semisimple} if it has pure global dimension zero. Equivalently, each of its objects is a direct summand of a coproduct of finitely presented objects. A ring $\Lambda$ is called {\it left pure semisimple} if the category $\Lambda$-Mod is pure semisimple. \\

Let $\mathscr{A}$ be a locally finitely presented Grothendieck category and suppose that every finitely presented object in $\mathscr{A}$ is a finite coproduct of indecomposable objects in $\mathscr{A}$. Let $\lbrace U_{\alpha}~|~ \alpha \in J \rbrace$ be a complete set of non-isomorphic finitely presented indecomposable objects in $\mathscr{A}$. Following \cite{y}, $\widehat{{\rm End}}_{\mathscr{A}}(\bigoplus _{\alpha \in J} U_{\alpha})$ is called {\it functor ring {\rm(or} Gabriel ring{\rm )} of $\mathscr{A}$} and we denote it by $R_{\mathscr{A}}$. By using \cite[Lemma 1.1]{aw}, we can see that $\widehat{{\Hom}}_{\mathscr{A}}(U,-): \mathscr{A} \rightarrow {\rm Flat}(R_{\mathscr{A}})$ is an additive equivalence which induces an equivalence between ${\rm fp}(\mathscr{A})$ and ${\rm proj}(R_{\mathscr{A}})$, where ${\rm Flat}(R_{\mathscr{A}})$ is the full subcategory of $R_{\mathscr{A}}{\rm Mod}$ consisting of flat unitary left $R_{\mathscr{A}}$-modules. By \cite[Theorem 3.4]{zn}, the functor ring of $\mathscr{A}$ uniquely determines up to Morita equivalence. Note that if $\Lambda$ is a left artinian ring (resp., right artinian), then we denoted the functor ring of $\Lambda$-Mod (resp., Mod-$\Lambda$) by $R_{\Lambda}$ (resp., $R_{{\Lambda}^{\rm op}}$) and we call it {\it left functor ring}  of $\Lambda$ (resp., {\it right functor ring} of $\Lambda$).

\begin{Cor}\label{cc1}
Let $\Lambda$ be a left pure semisimple ring. Then the left functor ring $R_{\Lambda}$ of $\Lambda$ has the properties $(i)$-$(iv)$.
\end{Cor}
\begin{proof}
By Proposition \ref{r4}, $R_{\Lambda}$ has the properties $(i)$ and $(iii)$-$(iv)$. Also by \cite[Theorem 2]{fs} and \cite[Lemma 2.2]{fin}, ${\rm l.gl.dim}R_{\Lambda} \leq 2$.  Moreover, if we had ${\rm l.gl.dim}R_{\Lambda} = 1$, then every finitely generated submodule of a projective unitary left $R_{\Lambda}$-module is projective. By \cite[Proposition 52.8]{wi}, $\Lambda$ is a left semisimple ring and so $R_{\Lambda}$ is Morita equivalent to $\Lambda$. This yields that every unitary left $R_{\Lambda}$-module is projective which is a contradiction. This proves that ${\rm l.gl.dim}R_{\Lambda}$ is either 0 or 2. Also by Proposition \ref{r4}(a), $R_{\Lambda}$ has the properties $(ii)$.
\end{proof}

Recall that a Grothendieck category $\mathscr{A}$ is called {\it perfect} if each of its objects admits a projective cover (see \cite[p. 334]{ha1}). A ring $R$ with enough idempotents is called {\it left {\rm (resp., }right{\rm)} perfect} if the category $R$Mod (resp., Mod$R$) is a perfect category.

\begin{Pro}\label{c4}
Let $R$ be a ring with the properties $(i)$-$(iv)$ and $S={\rm End}_R(P)$, where $P=\bigoplus_{i=1}^nP_i$ and each $P_i$ is a projective cover of $S_i$. Then
\begin{itemize}
\item[$(a)$] ${\rm Hom}_R(P,-):{\rm Proj}(R) \rightarrow S$-{\rm Mod} is an equivalence of categories.
\item[$(b)$] $S$ is a left pure semisimple ring.
\end{itemize}
\end{Pro}
\begin{proof}
$(a)$. The proof is given in several steps.\\

\textbf{Step 1.} We show that $S$ is a left artinian ring. First we show that $S$ is a right perfect ring. We know that each $P_i \cong Re_{\alpha_i}$ as $R$-modules for some $\alpha_i \in J$. Also, by the property $(iv)$, there exists a minimal injective resolution $0 \rightarrow Re_{\alpha_i} \rightarrow E_0 \rightarrow E_1$ where $E_0$ and $E_1$ are projective and injective unitary left $R$-modules. Since for each $l \in \lbrace 0, 1\rbrace$, $E_l$ is projective unitary left $R$-module, $E_l \cong \bigoplus_{j \in B_l}Re_{\alpha_j}$. Since it is injective in $R$Mod,  each $Re_{\alpha_j} \in {\rm proj}(R) \cap {\rm inj}(R)$. By the property $(iii)$, each $Re_{\alpha_j} \in {\rm add}(\bigoplus_{i=1}^nE(S_i))$ and so $E_l \in {\rm Add}(\bigoplus_{i=1}^nE(S_i))$. By Lemma \ref{c1}, each $E_l \in R{\rm Hom}_S(G_P, S{\rm Mod})$. Because $\gamma=(\gamma_M): id_{R{\rm Mod}} \rightarrow R{\rm Hom}_S(G_P, S{\rm Hom}_R(P,-))$ is a natural transformation and $0 \rightarrow Re_{\alpha_i} \rightarrow E_0 \rightarrow E_1$ is an exact sequence, the canonical morphism $\gamma_{Re_{\alpha_i}}: Re_{\alpha_i} \rightarrow R{\rm Hom}_S(G_P,S{\rm Hom}_R(P,Re_{\alpha_i}))$ is an isomorphism. Hence each $P_i \in R{\rm Hom}_S(G_P, S{\rm Mod})$. Let $\pi_i: P \rightarrow P_i$ and $\varepsilon_i: P_i \rightarrow P$ be the canonical projection and injection, respectively. Set $f_i=\pi_i\varepsilon_i$.  Then it is easy to see that $f_1, \cdots ,f_n$ are orthogonal idempotents in $S$ with $1_S= f_1 + \cdots + f_n$. Since ${\rm Hom}_R(P, P_i) \cong Sf_i$ as $S$-modules, by Corollary \ref{b3}, ${\rm End}_R(P_i) \cong Sf_iS$ as rings and so by the property $(i)$, each $Sf_iS$ is a right perfect ring. By \cite[Ex. 43.12 (2)]{wi}, $S$ is a right perfect ring. Hence by \cite[Proposition 43.9]{wi}, the descending chain condition for cyclic left ideals of $S$ holds. It follows by \cite[Proposition 31.8]{wi} that the descending chain condition for finitely generated left ideals of $S$ also holds. This implies that $S$ is a left artinian ring if $S$ is left noetherian. Therefore it is enough to show that $S$ is a left noetherian ring. Consider the ascending chain $I_1 \subseteq I_2 \subseteq \cdots$ of left ideals of $S$. Then we have an ascending chain $PI_1 \subseteq PI_2 \subseteq \cdots$ of $R$-submodules of $P$. Since $R$ is a left locally noetherian ring and $P$ is a finitely generated unitary left $R$-module, the chain of submodules of $P$ becomes stationary after finitely many steps. Since $P$ is a finitely generated projective unitary left $R$-module, by \cite[Proposition 18.4(3)]{wi}, we can see that the chain  $I_1 \subseteq I_2 \subseteq \cdots$ becomes stationary after finitely many steps. This yields that $S$ is a left noetherian ring.\\

\textbf{Step 2.} Let $S'$ be a simple unitary left $S$-module. We show that ${\rm Soc}(R{\rm Hom}_S(G_P,E(S')))$ is non-zero. Let $g: {\rm Hom}_R(P, Re_{\alpha_i}) \rightarrow S'$ be a projective cover of $S'$. By using \cite[Lemma 2.7]{zn}, it is not difficult to see that $\varphi:= R{\rm Hom}_S(G_P,{\rm Hom}_R(P,\pi_i))R{\rm Hom}_S(G_P,g)$ is a non-zero morphism, where $\pi_i:R \rightarrow Re_{\alpha_i}$ is the canonical projection. We can certainly assume that $J(R)R{\rm End}_S(G_P) \subseteq {\rm Ker}\varphi$, since otherwise there exist $\delta \in R{\rm End}_S(G_P)$ and $r \in J(R)$ such that $(r\delta)\varphi \neq 0$. This yields ${\rm Im}~r\delta {\rm Hom}_R(P,\pi_i) \nsubseteq {\rm Ker}g$ and so $r\delta {\rm Hom}_R(P,\pi_i)$ is a split epimorphism because ${\rm Ker}g$ is a maximal small submodule of ${\rm Hom}_R(P,Re_{\alpha_i})$ and ${\rm Hom}_R(P,Re_{\alpha_i})$ is a projective left $S$-module. Hence there is an $S$-module homomorphism $h: {\rm Hom}_R(P,Re_{\alpha_i}) \rightarrow G_P$ such that $hr\delta {\rm Hom}_R(P,\pi_i)=id_{{\rm Hom}_R(P,Re_{\alpha_i})}$. So ${\rm Hom}_R(P,\pi_i)hr\delta {\rm Hom}_R(P,\pi_i\varepsilon_i)={\rm Hom}_R(P,\pi_i\varepsilon_i)$, where $\varepsilon_i:Re_{\alpha_i} \rightarrow R$ is the canonical injection. On the other hand, since $R=\bigoplus_{\alpha \in J}Re_{\alpha}$ has the property $(iv)$ and $R$ is left locally noetherian, by Lemma \ref{c1}, we can see that $R$ is injectively copresented over $P$. Hence by Proposition \ref{b2}, there exists a ring isomorphism $R \rightarrow R{\rm End}_S(G_P)$ via $s \mapsto {\rm Hom}_R(P, \rho_s)$, where the mapping $\rho_s \in R{\rm End}_R(R)$ defined by $(x)\rho_s=xs$ for each $x \in R$. Since $r \in J(R)$, ${\rm Hom}_R(P, \rho_r) \in J(R{\rm End}_S(G_P))$ and so $r \delta={\rm Hom}_R(P, \rho_r)\delta \in J(R{\rm End}_S(G_P))$. Therefore ${\rm Hom}_R(P,\pi_i\varepsilon_i) \in J(R{\rm End}_S(G_P))$ which is a contradiction. Hence  ${\rm Im}\varphi$ is a semisimple unitary left $R$-module and so ${\rm Soc}(R{\rm Hom}_S(G_P,E(S')))$ is non-zero. \\

\textbf{Step 3.} Let $E$ be an injective left $S$-module. We show that $R{\rm Hom}_S(G_P,E)$ is a projective unitary left $R$-module. By using \cite[Lemma 2.7]{zn} and Proposition \ref{b2}, we can easily see that $R{\rm Hom}_S(G_P,E)$ is an injective module in $R$Mod. Since $R$ is a left locally noetherian ring, by \cite[Ch.  V, Proposition 4.5]{bs}, $R{\rm Hom}_S(G_P,E) = \bigoplus_{j \in A} N_j$, where each $N_j$ is an indecomposable injective module in $R$Mod. It is enough to show that each $N_j$ is a projective unitary left $R$-module.  We know that each $N_j$ is injectively copresented over $P$. Hence by Lemma \ref{c1} and the fact that $R$ has the property $(iii)$, it is sufficient to show that each $N_j$ has a simple essential socle. By using \cite[Lemma 2.7]{zn}, we can see that $E \cong \bigoplus_{j \in A}{\rm Hom}_R(P,N_j)$ as $S$-modules. Since each $N_j$ is injectively copresented over $P$, by Corollary \ref{b3} and \cite[Ch.  V, Proposition 5.1]{bs}, each ${\rm Hom}_R(P,N_j)$ is an indecomposable injective left $S$-module. By Step 1, $S$ is a left artinian ring and so for each $j \in A$, ${\rm Hom}_R(P, N_j) \cong E(S'_j)$ as $S$-modules for some simple left $S$-module $S'_j$. Since $N_j \cong R{\rm Hom}_S(G_P, {\rm Hom}_R(P, N_j)) \cong R{\rm Hom}_S(G_P, E(S'_j))$ as $S$-modules, by Step 2, ${\rm Soc}(N_j) \neq 0$ for each $j \in A$. Each $N_j$ is an indecomposable injective module and hence has a simple essential socle. \\

\textbf{Step 4.} By Corollary \ref{b3}, the functor ${\rm Hom}_R(P,-):R{\rm Hom}_S(G_P,S{\rm Mod}) \rightarrow S$-Mod is an equivalence of categories. It is enough to show that $R{\rm Hom}_S(G_P,S{\rm Mod})={\rm Proj}(R)$. Let $Q$ be a projective unitary left $R$-module. Since $R$ has the properties $(i)$ and $(iv)$, we have an exact sequence $0 \rightarrow Q \rightarrow E_0 \rightarrow E_1$ where each $E_l \in {\rm Proj}(R)$. $R$ is a left locally noetherian ring and so by \cite[Ch.  V, Proposition 4.3]{bs}, each $E_l \in {\rm Inj}(R)$. Hence by using the properties $(i)$ and $(iii)$ each $E_l \in {\rm Add}(\bigoplus_{i=1}^nE(S_i))$. By using Lemma \ref{c1}, it is not difficult to show that each $E_l$ is injectively copresented over $P$ and so by the five lemma $Q$ is injectively copresented over $P$. Now it is enough to show that each object in $R{\rm Hom}_S(G_P,S{\rm Mod})$ is a projective unitary left $R$-module. Let $M$ be injectively copresented over $P$. Then $M \cong R{\rm Hom}_S(G_P, N)$ as $R$-modules for some left $S$-module $N$. Consider the following exact sequence $$0 \rightarrow M \rightarrow R{\rm Hom}_S(G_P, I_0) \rightarrow R{\rm Hom}_S(G_P, I_1),$$
where $0 \rightarrow N \rightarrow I_0 \rightarrow I_1$ is the minimal injective resolution of $N$. By Step 3, each $R{\rm Hom}_S(G_P, I_i)$ is a projective unitary left $R$-module. Since ${\rm l.gl.dim}(R) =0 ~{\rm or}~ 2$, $M$ is a projective unitary left $R$-module. Therefore $R{\rm Hom}_S(G_P,S{\rm Mod})={\rm Proj}(R)$ and the result follows.\\

$(b)$. Since $R$ has the property $(i)$, every projective unitary left $R$-module is a direct sum of finitely generated unitary module with local endomorphism ring. By $(a)$, every left $S$-module is a direct sum of indecomposable left $S$-module. Therefore by \cite[Corollary 2.7]{sp}, $S$ is a left pure semisimple ring.
\end{proof}

We are now in a position to prove our main result in this section.

\begin{The}\label{re1}
There exists a bijection between Morita equivalence classes of left pure semisimple rings $\Lambda$ and Morita equivalence classes of rings $R$ with the properties $(i)$-$(iv)$.
\end{The}
\begin{proof}
If $\Lambda$ is a left pure semisimple ring, then by Corollary \ref{cc1}, $R_{\Lambda}$ has the properties $(i)$-$(iv)$. Let $\Lambda$ and $\Lambda'$ be two left pure semisimple rings. If $\Lambda$ and $\Lambda'$ are Morita equivalent, then by using \cite[Theorem 3.4]{zn}, we can see that $R_{\Lambda}$ and $R_{\Lambda'}$ are Morita equivalent. Also, it is clear that if $R_{\Lambda}$ and $R_{\Lambda'}$ are Morita equivalent, then $\Lambda$ and $\Lambda'$ are Morita equivalent. This means that the assignment $\Lambda \mapsto R_{\Lambda}$ sends Morita equivalence classes of left pure semisimple rings to Morita equivalence classes of rings with the properties $(i)$-$(iv)$. Moreover, this shows that this map is injective. Now we show that this map is surjective. Let $R$ be a ring with the properties $(i)$-$(iv)$. Then by Proposition \ref{c4}, there exists an additive equivalence between ${\rm Proj}(R) \rightarrow S$-Mod, where $S$ is a left pure semisimple ring. Hence we have an additive equivalence ${\rm Proj}(R) \rightarrow {\rm Proj}(R_{S})$ which preserves and reflects finitely generated modules. By \cite[Theorem 3.4]{zn}, $R$ and $R_S$ are Morita equivalence and so this map is surjective. Therefore this map is a bijection between Morita equivalence classes of left pure semisimple rings and Morita equivalence classes of rings with the properties $(i)$-$(iv)$.
\end{proof}

As a consequence of Theorem \ref{re1}, and \cite[Theorem 3.1]{wis}, we have the following corollaries.

\begin{Cor}{\rm (See \cite[Corollary 4.7]{Ausla2})}\label{re2}
There exists a bijection between Morita equivalence classes of rings of finite representation type and Morita equivalence classes of left artinian rings $R$ with ${\rm l.gl.dim}R =0 ~{\rm or}~ 2$ and ${\rm l. dom.dom}{R} \geq 2$.
\end{Cor}

\begin{Cor}\label{re8}
The following statements are equivalent.
\begin{itemize}
\item[$(a)$] The pure semisimplicity conjecture holds.
\item[$(b)$] Every ring with the properties $(i)$-$(iv)$ is Morita equivalent with an Auslander ring.
\end{itemize}
\end{Cor}

\section{Ringel-Tachikawa Correspondence for left pure semisimple rings}
In this section we extend the Ringel-Tachikawa correspondence to the class of left pure semisimple rings. Let $\Lambda$ be a basic left pure semisimple ring. We show that the right functor ring of the endomorphism ring of minimal injective cogenerator of $\Lambda$-{\rm Mod} has nice properties (see Proposition \ref{l7}). Then we show that these properties characterize left pure semisimple rings up to Morita equivalence (see Theorem \ref{rt1}). This result extends the Ringel-Tachikawa correspondence to the class of left pure semisimple rings. As a consequence we provide another equivalent statement for the pure semisimplicity conjecture (see Corollary \ref{rtt2}).\\

Let $R$ be a ring and $N$ be a unitary left $R$-submodule of a module $M$. Then $M$ is called a {\it rational extension} of $N$ if for every module $D$ with $N \subseteq D \subseteq M$ and every $R$-homomorphism $f: D \rightarrow M$, the inclusion $N \hookrightarrow {\rm Ker}f$ implies $f = 0$ (see \cite[p. 58]{ca}).  Assume that $R$ is a subring of a ring $Q$. Then $Q$ is called a {\it left quotient ring} of $R$ if $Q$ is a rational extension of $R$ as left $R$-modules (see \cite[p. 64]{ca}). \\

Let $\mathscr{A}$ be a Grothendieck category with a generating set $\lbrace P_{\alpha}~|~\alpha \in J\rbrace$ of projective objects. We recall from \cite[p. 359]{ha3} that an object $M$ in $\mathscr{A}$ is called {\it faithful} if for any non-zero morphism $f: P_{\alpha} \rightarrow P_{\beta}$ there exists a morphism $g: P_{\beta} \rightarrow M$ such that $fg \neq 0$. Harada in \cite{ha3} showed that the faithfulness of $M$ dose not depend on generating sets of projectives. Let $R$ be a ring with enough idempotents. It is easy to see that a unitary left $R$-module $M$ is a faithful object in $R$Mod if and only if ${\rm ann}_R(M)=0$ (i.e., $M$ is a faithful left $R$-module).

\begin{Lem}\label{l1}
Let $R$ be a ring with enough idempotents and $fR$, $f^2=f \in R$, be a faithful right ideal of $R$. Then
\begin{itemize}
\item[$(a)$] $R{\rm End}_{fRf}(fR)$ is a left quotient ring of $R$.
\item[$(b)$] If $Re$, $e^2=e \in R$, is an injective module in $R$Mod, then $R{\rm Hom}_{fRf}(fR, fRe) \cong Re$ as $R$-modules. In particular, $fRe$ is an injective left $fRf$-module.
\end{itemize}
\end{Lem}
\begin{proof}
$(a)$. Since $fR$ is faithful, the canonical ring homomorphism $\rho: R \rightarrow R{\rm End}_{fRf}(fR)$ via $r \mapsto [\rho_r: x \mapsto xr]$ is a monomorphism. If $d \in R{\rm End}_{fRf}(fR)$ and $x \in fR$ such that $(x)d \neq 0$, then $\rho_xd=\rho_{(x)d} \neq 0$. Hence the result holds.\\
$(b)$. It is easy to see that the mapping $\varphi: R{\rm Hom}_{fRf}(fR, fRe) \rightarrow R{\rm End}_{fRf}(fR)e$ via $\alpha \mapsto \alpha \varepsilon e$ is an isomorphism of $R$-modules, where $\varepsilon: fRe \rightarrow fR$ is the canonical injection. Also since $e$ is an idempotent of $R$ and $fR$ is faithful, we can see that $Re \cong {\rm Im}\rho|_{Re} = {\rm Im}\rho e$ as $R$-modules. But because $R{\rm End}_{fRf}(fR)$ is a left quotient ring of $R$, ${\rm Im}\rho$ is an essential ${\rm Im}\rho$-submodule of $R{\rm End}_{fRf}(fR)$. This implies that ${\rm Im}\rho e$ is an essential $R$-submodule of $R{\rm End}_{fRf}(fR)e$  and so ${\rm Im}\rho e = R{\rm End}_{fRf}(fR)e$ because $Re$ is injective in $R$Mod. Therefore $R{\rm Hom}_{fRf}(fR, fRe) \cong Re$ as $R$-modules. Since $fR$ is a projective unitary right $R$-module, the functor $fR \otimes_R-$ sends $R$-module monomorphisms to $fRf$-module monomorphisms. Since $Re$ is injective module in $R$Mod and the natural transformation
\begin{center}
$\eta: fR \otimes_R R{\rm Hom}_{fRf}(fR,-) \rightarrow {\rm id}$ via $fr \otimes \alpha \mapsto (fr)\alpha$
\end{center}
for $r \in R$, $\alpha \in R{\rm Hom}_{fRf}(fR,X)$, where $X$ is a left $fRf$-module, is an equivalence, $fRe$ is an injective left $fRf$-module. Therefore the result follows.
\end{proof}

Let $\mathscr{A}$ be a Grothendieck category with a generating set of projective objects. An object $M$ in $\mathscr{A}$ is called a {\it minimal faithful} if $M$ is faithful and for every faithful object $N$ in $\mathscr{A}$, there exists an object $K$ in $\mathscr{A}$ such that $N \cong M \bigoplus K$ (see \cite[p. 360]{ha3}).  Fuller in \cite[Proposition 3.2]{f} showed that Colby and Rutter's characterization of the minimal faithful modules over a ring with identity \cite{cr} generalizes to a ring with enough idempotents. He proved a unitary left $R$-module $M$ is finitely generated minimal faithful if and only if there are finitely many pairwise non-isomorphic simple left ideals $\lbrace S_1, \cdots, S_n\rbrace$ such that $M \cong \bigoplus_{i=1}^n E(S_i)$ is faithful, projective, and injective in $R$Mod.

Thus, the finitely generated minimal faithful unitary left module of a semiperfect ring $R$ with enough idempotents is always isomorphic to a left ideal of the form $Re=\oplus_{i=1}^nRe_i$ where each $Re_i$ is injective in $R$Mod and it has simple essential socle and for each $i \neq j$, $Re_i \ncong Re_j$ as $R$-modules, where each $e_i$ is a local orthogonal idempotent in $R$ and $e=e_1 + \cdots +e_n$.

\begin{Pro}\label{l2}
Let $R$ be a semiperfect ring with enough idempotents and assume that $fR=\oplus_{j=1}^mf_jR$ and $Re = \oplus_{i=1}^nRe_i$ are minimal faithful right and left ideals of $R$, respectively. Then
\begin{itemize}
\item[$(a)$] $n=m$, ${\rm Soc}(Re_i) \cong {\rm top}(Rf_{\sigma(i)})$ and ${\rm Soc}(f_{\sigma(i)}R) \cong {\rm top}(e_iR)$ as $R$-modules, where $\sigma$ is a permutation of $\lbrace 1, \cdots, n \rbrace$.
\item[$(b)$] $fRe$ is a cogenerator in $fRf$-Mod.
\end{itemize}
\end{Pro}
\begin{proof}
$(a)$. Since each $Re_i$ has a simple essential socle, ${\rm Soc}(Re_i)=Ra_i$ for some $a_i \in R$. Since $Re_i$ is injective in $R$Mod, $a_iR$ is a unique simple submodule of ${\rm r.ann}_R({\rm l.ann}_R(a_i))$. Since $a_i \in Re_i$, the mapping $\varphi: e_iR \rightarrow a_iR$ via $e_ir \mapsto a_ir$ is well-defined and it is easy to see that it is an $R$-module epimorphism. Since $a_iR$ is simple and $e_i$ is a local idempotent, ${\rm Ker}\varphi={\rm rad}(e_iR)$ and so ${\rm top}(e_iR) \cong a_iR$ as $R$-modules. Similarly we can see that ${\rm Soc}(f_jR)=b_jR$ for some $b_j \in R$, $Rb_j$ is the unique simple submodule of ${\rm l.ann}_R({\rm r.ann}_R(b_j))$ and ${\rm top}(Rf_j) \cong Rb_j$ as $R$-modules. Now we show that for each $i,~(1 \leq i \leq n)$, there exists $1 \leq j \leq m$ such that ${\rm Soc}(Re_i) \cong {\rm top}(Rf_j)$ and ${\rm Soc}(f_jR) \cong {\rm top}(e_iR)$ as $R$-modules. Since $fR$ is faithful, there is an $R$-module monomorphism $R \rightarrow \widehat{\prod}_A fR$, where $A$ is a set. This implies that there is an $R$-module homomorphism $\beta: R \rightarrow fR$ such that $\beta|_{a_iR} \neq 0$. Since $a_iR$ is simple, $\beta|_{a_iR}: a_iR \rightarrow fR$ is a monomorphism. Since $fR$ is injective in Mod$R$, there is an $R$-module monomorphism $E(a_iR) \rightarrow fR$. This yields that $E(a_iR) \cong f_jR$ for some $1 \leq j \leq m$. Therefore $a_iR \cong {\rm Soc}(f_jR)$ as $R$-modules and so ${\rm top}(e_iR) \cong {\rm Soc}(f_jR)$ as $R$-modules.  Now we show that ${\rm Soc}(Re_i) \cong {\rm top}(Rf_j)$ as $R$-modules. Since $a_iR \cong {\rm Soc}(f_jR)=b_jR$ as $R$-modules, $R{\rm Hom}_R(a_iR, R) \cong R{\rm Hom}_R(b_jR, R)$ as $R$-modules. It is not difficult to see that the mapping $R{\rm Hom}_R(a_iR, R) \rightarrow {\rm l.ann}_R({\rm r.ann}_R(a_i))$ via $\alpha \mapsto \alpha(a_i)$ is an $R$-module isomorphism. Similarly $R{\rm Hom}_R(b_jR, R) \cong {\rm l.ann}_R({\rm r.ann}_R(b_j))$ as $R$-modules. Therefore ${\rm l.ann}_R({\rm r.ann}_R(b_j)) \cong {\rm l.ann}_R({\rm r.ann}_R(a_i))$ as $R$-modules. This yields $Rb_j \cong {\rm Soc}({\rm l.ann}_R({\rm r.ann}_R(a_i)))$ as $R$-modules. Since $Ra_i {\rm r.ann}_R(a_i)=0$, $Ra_i \subseteq {\rm l.ann}_R({\rm r.ann}_R(a_i))$. Since $Ra_i$ is simple, ${\rm Soc}(Re_i) \cong {\rm top}(Rf_j)$ as $R$-modules. Consequently for each $i,~(1 \leq i \leq n)$, there is  $1 \leq j \leq m$, such that ${\rm Soc}(Re_i) \cong {\rm top}(Rf_j)$ and ${\rm Soc}(f_jR) \cong {\rm top}(e_iR)$ as $R$-modules.  Similarly we can see that for each $l,~(1 \leq l \leq m)$, there is $1 \leq k \leq n$ such that ${\rm Soc}(Re_k) \cong {\rm top}(Rf_l)$ and ${\rm Soc}(f_lR) \cong {\rm top}(e_kR)$ as $R$-modules. This shows that the result follows. \\

$(b)$. Since $Re=\bigoplus_{i=1}^nRe_i$, $fRe=\bigoplus_{i=1}^nfRe_i$. By Lemma \ref{l1}, $fRe$ is an injective left $fRf$-module and hence  each $fRe_i$ is an injective left $fRf$-module. Also by using Lemma \ref{l1}, it is not difficult to see that ${\rm End}_R(Re_i) \cong {\rm End}_{fRf}(fRe_i)$ as rings. Since ${\rm End}_R(Re_i)$ is a local ring, ${\rm End}_{fRf}(fRe_i)$ is a local ring and so each $fRe_i$ is an indecomposable injective left $fRf$-module.  Now we show that ${\rm Soc}(fRe_i) \neq 0$. By $(a)$, ${\rm Soc}(Re_i) \cong {\rm top}(Rf_{\sigma(i)})$ as $R$-modules, where $\sigma$ is a permutation of $\lbrace 1, \cdots, n \rbrace$. Hence $f{\rm Soc}(Re_i) \cong f{\rm top}(Rf_{\sigma(i)})$ as $fRf$-modules. If $f{\rm Soc}(Re_i)=0$, then $f{\rm top}(Rf_{\sigma(i)})=0$ and so $f_{\sigma(i)} \in {\rm rad}(Rf_{\sigma(i)})$. Hence ${\rm Soc}(Re_i)=0$ which is a contradiction. Thus $f{\rm Soc}(Re_i)\neq 0$ and we can easily see that it is a simple $fRf$-submodule of $fRe_i$. Therefore ${\rm Soc}(fRe_i) \neq 0$. But because $fRe_i$ is an indecomposable injective left $fRf$-module, $fRe_i$ has a simple essential socle. On the other hand, since ${\rm Soc}(fRe)=\bigoplus_{i=1}^n{\rm Soc}(fRe_i)$, it is enough to show that $\lbrace {\rm Soc}(fRe_1), \cdots, {\rm Soc}(fRe_n)\rbrace$ is a complete set of non-isomorphic simple left $fRf$-modules. Since $\lbrace fRf_1/{\rm rad}(fRf_1), \cdots, fRf_n/{\rm rad}(fRf_n)\rbrace$ is a complete set of non-isomorphic simple left $fRf$-modules, it is sufficient to show that ${\rm Soc}(fRe_i) \cong fRf_{\sigma(i)}/{\rm rad}(fRf_{\sigma(i)})$ as $fRf$-modules. It is not difficult to see that $f{\rm top}(Rf_{\sigma(i)}) \cong fRf_{\sigma(i)}/f{\rm rad}(Rf_{\sigma(i)})$ as $fRf$-modules. But since $f{\rm top}(Rf_{\sigma(i)}) \cong f{\rm Soc}(Re_i)={\rm Soc}(fRe_i)$ as $fRf$-modules and $fRe_i$ has a simple socle, ${\rm rad}(fRf_{\sigma(i)}) = f{\rm rad}(Rf_{\sigma(i)})$. This implies that $f{\rm top}(Rf_{\sigma(i)}) \cong {\rm top}(fRf_{\sigma(i)})$ as $fRf$-modules. Therefore ${\rm Soc}(fRe_i) \cong {\rm top}(fRf_{\sigma(i)})$ as $fRf$-modules and the result follows.
\end{proof}

Let $R$ be a ring with enough idempotents and $M$ be a unitary right $R$-module. According to \cite[p. 50]{rt}, the module $M$ is called {\it balanced} if the canonical ring homomorphism $\rho: R \rightarrow R{\rm End}_{S}(M)$ via $r \mapsto [\rho_r: x \mapsto xr]$ is surjective, where $S={\rm End}_R(M)$. The ring $R$ is called {\it right locally coherent} if every direct product of flat unitary left $R$-modules in $R$Mod is flat (see \cite[p. 118]{gs}). \\

Let $\mathscr{A}$ be a Grothendieck category with a generating set of projective objects. According to \cite[p. 102]{f}, $\mathscr{A}$ is called {\it QF-3} if it contains a finitely generated minimal faithful object. A ring $R$ with enough idempotents is called {\it left QF-3} if the category $R$Mod is a QF-3 category. Right and two-sided QF-3 rings defined analogously.

\begin{Pro}\label{l5}
Let $R$ be a left perfect right locally coherent QF-3 ring with enough idempotents such that ${\rm l.gl.dim}R = 0 ~{\rm or} ~ 2$. Assume that $fR$, $f^2=f \in R$, is a minimal faithful balanced right ideal of $R$. Then
\begin{itemize}
\item[$(a)$] $fRf$ is a left pure semisimple ring.
\item[$(b)$] The functor
\begin{center}
$R{\rm Hom}_{fRf}(fR,-): fRf$-Mod $\rightarrow {\rm Proj}(R)$
\end{center}
is an additive equivalence.
\end{itemize}
\end{Pro}
\begin{proof}
$(a)$. Let $M$ be a left $fRf$-module and $Re$, $e^2=e\in R$, be a minimal faithful left ideal of $R$. By Proposition \ref{l2}(b), $fRe$ is a cogenerator in $fRf$-Mod. Then there exists an exact sequence in $fRf$-Mod
$$0 \rightarrow M \rightarrow \prod_{I} {fRe} \rightarrow \prod_{J} {fRe},$$
 where $I$ and $J$ are two index sets. By using Proposition \ref{l1}, we can see that $$R{\rm Hom}_{fRf}(fR,\prod_{I} fRe) \cong \widehat{\prod_{I}} Re$$ as $R$-modules. This gives us the following exact sequence
 $$0 \rightarrow R{\rm Hom}_{fRf}(fR,M) \rightarrow \widehat{\prod_{I}}Re \rightarrow \widehat{\prod_{J}}Re$$
Since $R$ is a left perfect right locally coherent ring and ${\rm l.gl.dim}R =0 ~{\rm or}~ 2$, $R{\rm Hom}_{fRf}(fR,M) \in {\rm Proj}(R)$. Since $R$ is a semiperfect ring with enough idempotents, there is a complete set $\lbrace e_{\alpha}~|~\alpha \in J \rbrace$ of pairwise orthogonal local idempotents of $R$ such that $R=\bigoplus_{\alpha \in J}Re_{\alpha}=\bigoplus_{\alpha \in J}e_{\alpha}R$. Hence $R{\rm Hom}_{fRf}(fR,M) \cong \bigoplus_{\alpha \in A}Re_{\alpha}$ as $R$-modules. There are $fRf$-module isomorphisms
$$M \cong fR \otimes_R R{\rm Hom}_{fRf}(fR,M) \cong fR \otimes_R (\bigoplus_{\alpha \in A}Re_{\alpha}) \cong \bigoplus_{\alpha \in A} (fR \otimes_R Re_{\alpha}) \cong \bigoplus_{\alpha \in A} fRe_{\alpha}.$$
Since $fR$ is balanced and faithful, the canonical ring homomorphism $\rho: R \rightarrow R{\rm End}_{fRf}(fR)$ via $r \mapsto [\rho_r: x \mapsto xr]$ is an isomorphism.  Hence $Re_{\alpha} \cong R{\rm End}_{fRf}(fR)e_{\alpha}$ as $R$-module. It is easy to see that the mapping $\varphi: R{\rm Hom}_{fRf}(fR, fRe_{\alpha}) \rightarrow R{\rm End}_{fRf}(fR)e_{\alpha}$ via $\alpha \mapsto \alpha \varepsilon e_{\alpha}$ is an isomorphism of $R$-modules, where $\varepsilon: fRe_{\alpha} \rightarrow fR$ is the canonical injection. It follows that $R{\rm Hom}_{fRf}(fR, fRe_{\alpha}) \cong Re_{\alpha}$ as $R$-modules and so ${\rm End}_R(Re_{\alpha}) \cong {\rm End}_{fRf}(fRe_{\alpha})$ as rings. Since ${\rm End}_R(Re_{\alpha})$ is a local ring, ${\rm End}_{fRf}(fRe_{\alpha})$ is a local ring and so each $fRe_{\alpha}$ is an indecomposable left $fRf$-module. Therefore every left $fRf$-module is a direct sum of indecomposable modules. Consequently by \cite[Corollary 2.7]{sp}, $fRf$ is a left pure semisimple ring.\\

$(b)$. By using the proof of $(a)$, we can see that functor $R{\rm Hom}_{fRf}(fR, -)$ sends $fRf$-Mod to the full subcategory ${\rm Proj}(R)$ of $R$Mod. The natural transformation
\begin{center}
$\eta: fR \otimes_R R{\rm Hom}_{fRf}(fR,-) \rightarrow {\rm id}$ via $fr \otimes \alpha \mapsto (fr)\alpha$
\end{center}
for $r \in R$ and $\alpha \in R{\rm Hom}_{fRf}(fR,X)$, where $X$ is a left $fRf$-module, is an equivalence. Therefore it is enough to show that the natural transformation
\begin{center}
$\nu: {\rm id} \rightarrow R{\rm Hom}_{fRf}(fR, fR \otimes_R -)$ via $x \mapsto [fr \mapsto fr \otimes x]$
\end{center}
for $x \in X$ and $r \in R$ is an equivalence when $X$ is a projective unitary left $R$-module. Let $P$ be a  projective unitary left $R$-module. Then $P \cong \bigoplus_{\alpha \in A}Re_{\alpha}$ as $R$-modules. We know that for each idempotent $e'$ of $R$, $fRe'$ is a finitely generated left $fRf$-module, because $fRf$ is left pure semisimple. This implies that the functor $R{\rm Hom}_{fRf}(fR,-)$ preserves direct sums and so we have the following isomorphisms
\begin{center}
$R{\rm Hom}_{fRf}(fR, fR \otimes_R P) \cong R{\rm Hom}_{fRf}(fR, \bigoplus_{\alpha \in A} (fR \otimes_R Re_{\alpha}))\cong$\\ $\bigoplus_{\alpha \in A} R{\rm Hom}_{fRf}(fR, fR \otimes_R Re_{\alpha}) \cong \bigoplus_{\alpha \in A} R{\rm Hom}_{fRf}(fR, fRe_{\alpha}) \cong \bigoplus_{\alpha \in A}Re_{\alpha} \cong P$
\end{center}
which imply that $\nu_{P}$ is an isomorphism. Therefore the functor
\begin{center}
$R{\rm Hom}_{fRf}(fR,-): fRf$-Mod $\rightarrow {\rm Proj}(R)$
\end{center}
is an additive equivalence with the inverse equivalence $fR \otimes_R-: {\rm Proj}(R) \rightarrow fRf$-Mod.
\end{proof}

A left $\Lambda$-module $Q$ is called {\it minimal injective cogenerator} if it is isomorphic to $\bigoplus_{i=1}^nE(S_i)$, where $\lbrace S_1, \cdots, S_n\rbrace$ is a complete set of non-isomorphic simple left $\Lambda$-modules.

\begin{Lem}\label{l6}
Let $\Lambda$ be a right artinian ring and $\lbrace U_{\alpha}~|~\alpha \in I\rbrace$ be a complete set of non-isomorphic finitely generated indecomposable right $\Lambda$-modules. Assume that $W$ is a minimal injective cogenerator in {\rm Mod}-$\Lambda$ and set $U=\bigoplus_{\alpha \in I}U_{\alpha}$ and $R=\widehat{{\rm End}}_{\Lambda}(U)$. If $W$ is finitely generated, then $\widehat{{\rm Hom}}_{\Lambda}(U, W)$ is a minimal faithful balanced unitary right $R$-module.
\end{Lem}
\begin{proof}
Since $W$ is finitely generated, $W \in {\rm add}(U)$ and so $U$ is a generator-cogenerator in Mod-$\Lambda$. Hence by using the proof of Proposition \ref{r4}, we can see that $\widehat{{\rm Hom}}_{\Lambda}(U, W)$ is a minimal faithful unitary right $R$-module. Set $\Delta={\rm End}_{\Lambda}(W)$. Since $\widehat{{\rm Hom}}_{\Lambda}(U, W) \cong \bigoplus_{\alpha \in I}{\rm Hom}_{\Lambda}(U_{\alpha}, W)$ as $\Delta$-modules and the functor ${\rm Hom}_{\Lambda}(-, W):$ mod-$\Lambda \rightarrow \Delta$-mod is a duality, for each $\gamma \in\\ R{\rm End}_{\Delta}(\widehat{{\rm Hom}}_{\Lambda}(U, W))$, there exists $r \in R$ such that $\rho_r=\gamma$. This implies that the canonical ring homomorphism $\rho: R \rightarrow R{\rm End}_{\Delta}(\widehat{{\rm Hom}}_{\Lambda}(U, W))$  via $r \mapsto [\psi \mapsto \psi r]$ is surjective.
\end{proof}

An idempotent $e$ of a semiperfect ring $\Lambda$ is called {\it basic idempotent} in case $e$ is the sum $e= e_1 + \cdots + e_m$ of a basic set $\lbrace e_1, \cdots, e_m\rbrace$ of idempotents of $\Lambda$. A semiperfect ring $\Lambda$ is called {\it basic} if $1_{\Lambda}$ is a basic idempotent of $\Lambda$ (see \cite[Sect. 27]{an}).

\begin{Pro}\label{l7}
Let $\Lambda$ be a left pure semisimple basic ring and suppose that $W$ is a minimal injective cogenerator in $\Lambda$-{\rm Mod}. Set $\Delta={\rm End}_{\Lambda}(W)$. Then the right functor ring $R_{{\Delta}^{op}}$ of $\Delta$ is a left perfect right locally coherent QF-3 ring with enough idempotents which has a minimal faithful balanced right ideal and ${\rm l.gl.dim}R_{{\Delta}^{op}} = 0 ~{\rm or} ~ 2$.
\end{Pro}
\begin{proof}
By Corollary \ref{cc1}, \cite[Lemma 2.3]{fin} and \cite[Theorem]{fu}, $R_{\Lambda}$ is a left perfect ring with enough idempotents such that ${\rm l.gl.dim}R_{\Lambda} = 0 ~{\rm or} ~ 2$. Since $W$ is Morita duality $\Lambda$-$\Delta$-bimodule (as discussed in \cite[Sect. 24]{an}), by \cite[Theorem 3.3]{f} and \cite[Lemma 2.3]{fin}, $R_{\Lambda}$ is a QF-3 ring. Moreover by \cite[Lemma 3.1]{f} and \cite[Lemma 2.3]{fin}, $R_{\Lambda}$ is Morita equivalent to $R_{{\Delta}^{op}}$. It is not difficult to see that the direct product of any family of projective unitary left $R_{\Lambda}$-modules is projective. This implies that by  \cite[Proposition 2.4]{gs}, $R_{\Lambda}$ is a right locally coherent ring. Therefore $R_{{\Delta}^{op}}$ is a left perfect QF-3 right locally coherent ring with enough idempotents with ${\rm l.gl.dim}R_{{\Delta}^{op}} =0 ~{\rm or}~ 2$. On the other hand, since $W$ is finitely generated minimal injective cogenerator in Mod-$\Delta$ and $\Delta$ is a right artinian ring, by Lemma \ref{l6}, $R_{{\Delta}^{op}}$ has a minimal faithful balanced right ideal. Therefore $R_{{\Delta}^{op}}$ is a left perfect right locally coherent QF-3 ring with enough idempotents which has a minimal faithful balanced right ideal and ${\rm l.gl.dim}R_{{\Delta}^{op}} = 0 ~{\rm or} ~ 2$.
\end{proof}

We are now in a position to prove our main result in this section.

\begin{The}\label{rt1}
There exists a bijection between Morita equivalence classes of left pure semisimple rings $\Lambda$ and Morita equivalence classes of left perfect right locally coherent QF-3 rings $R$ with enough idempotents which have a minimal faithful balanced right ideal and ${\rm l.gl.dim}R = 0 ~{\rm or} ~ 2$.
\end{The}
\begin{proof}
Let $\Lambda$ be a left pure semisimple ring. Then there exists a left pure semisimple basic ring $\Lambda'$ such that $\Lambda$ is Morita equivalent to $\Lambda'$. Assume that $W$ is a minimal injective cogenerator in $\Lambda'$-Mod and $\Delta={\rm End}_{\Lambda'}(W)$. Then by Proposition \ref{l7}, the right functor ring $R_{{\Delta}^{op}}$ of $\Delta$ is a left perfect right locally coherent QF-3 ring with enough idempotents which has a minimal faithful balanced right ideal and ${\rm l.gl.dim}R_{{\Delta}^{op}} = 0 ~{\rm or} ~ 2$. By using \cite[Theorem 3.4]{zn}, it is easy to see that the assignment $\Lambda \mapsto R_{{\Delta}^{op}}$ sends Morita equivalence classes of left pure semisimple rings to Morita equivalence classes of left perfect right locally coherent QF-3 rings $R_{{\Delta}^{op}}$ with enough idempotents which have a minimal faithful balanced right ideal and ${\rm l.gl.dim}R_{{\Delta}^{op}} = 0 ~{\rm or} ~ 2$. Also, this map is injective.  Now we show that this map is surjective. Let $R$ be a left perfect right locally coherent QF-3 ring with enough idempotents such that ${\rm l.gl.dim}R = 0 ~{\rm or} ~ 2$. Assume that $fR$, $f^2=f \in R$, is a minimal faithful balanced right ideal of $R$. By Proposition \ref{l5}, $fRf$ is a left pure semisimple ring and the functor $R{\rm Hom}_{fRf}(fR,-): fRf$-Mod $\rightarrow {\rm Proj}(R)$ is an additive equivalence. Let $Q$ be a minimal injective cogenerator in $fRf$-Mod and $\Gamma={\rm End}_{fRf}(Q)$. It is not difficult to show that $R_{{\Gamma}^{op}}$ is Morita equivalent to $R$. Hence this map is surjective and the result follows.
\end{proof}

Let $R$ be a ring and $Q$ be a left quotient ring of $R$. The ring $Q$ is called {\it maximal} in case given any left quotient ring $S$ of $R$, there is a ring monomorphism from $S$ to $Q$ which induces the identity morphism on $R$ (see \cite[p. 66]{ca}). It is well-known that every ring with enough idempotents has a maximal left quotient ring which is a ring with identity (see \cite[p. 68, Proposition F]{ca}). By using \cite[Proposition 1.8]{rt} and \cite[Theorems 1.2 and 1.3]{cr1}, we can see that a ring with identity $\Gamma$ is a semiprimary QF-3 maximal quotient ring with ${\rm l.gl.dim}\Gamma = 0 ~{\rm or} ~ 2$  if and only if $\Gamma$ is a left perfect right locally coherent QF-3 ring which has a minimal faithful balanced right ideal and ${\rm l.gl.dim}\Gamma = 0 ~{\rm or} ~ 2$. Therefore we have the following corollaries.

\begin{Cor}{\rm (See \cite[Theorem 4.3]{rt})}\label{rt2}
There exists a bijection between Morita equivalence classes of rings of finite representation type $\Lambda$ and Morita equivalence classes of semiprimary QF-3 maximal quotient rings $R$ with ${\rm l.gl.dim}R = 0 ~{\rm or} ~ 2$.
\end{Cor}

\begin{Cor}\label{rtt2}
The following statements are equivalent.
\begin{itemize}
\item[$(a)$] The pure semisimplicity conjecture holds.
\item[$(b)$] Every left perfect right locally coherent QF-3 ring $R$ with enough idempotents which have a minimal faithful balanced right ideal and ${\rm l.gl.dim}R = 0 ~{\rm or} ~ 2$ is Morita equivalent to a maximal quotient ring.
\end{itemize}
\end{Cor}

\section{pure semisimple Grothendieck categories}
In this section we study locally finitely presented pure semisimple Grothendieck categories. First we show that the functor ring of any locally finitely presented pure semisimple Grothendieck category
which satisfies the property $(\ast)$, has the properties $(i)$-$(iv)$ (see Proposition \ref{r51}). Then we show that for any ring $R$ which has the properties $(i)$-$(iv)$, ${\rm Flat}(R)$ is a locally finitely presented pure semisimple Grothendieck category which satisfies the property $(\ast)$ (see Proposition \ref{r52}). By using this facts we give a bijection between Morita equivalence classes of left pure semisimple (hereditary)
rings and equivalence classes of locally finitely presented pure semisimple (hereditary) Grothendieck
categories which satisfy the property $(\ast)$ (see Corollaries \ref{r54} and \ref{r64}). Finally we provide another equivalent statement for the pure semisimplicity conjecture (see Corollary \ref{p22}).\\

Let $\mathscr{A}$ be a locally finitely presented category. Then by \cite[Theorem 1.4]{wc}, ${\rm fp}(\mathscr{A})$ is a skeletally small additive category with split idempotents, and the functor
\begin{center}
$\mathbb{H}: \mathscr{A} \rightarrow {\rm Mod}({\rm fp}(\mathscr{A}))$ via $A \mapsto {\rm Hom}_{\mathscr{A}}(-,A)|_{{\rm fp}(\mathscr{A})}$
\end{center}
induces an additive equivalence between the category $\mathscr{A}$ and the category ${\rm Flat}({\rm fp}(\mathscr{A}))$.

Now let $\lbrace V_{\alpha}~|~ \alpha \in I \rbrace$ be a complete set of non-isomorphic finitely presented objects in $\mathscr{A}$. Set $T_{\mathscr{A}}= \widehat{{\rm End}}_{\mathscr{A}}(V)$, where $V= \bigoplus_{\alpha \in I}V_\alpha$. It is shown in \cite[Proposition 2, p. 347]{gab} that the additive functor
\begin{center}
$\mathbb{G}: {\rm Mod}({\rm fp}(\mathscr{A})) \rightarrow T_{\mathscr{A}}{\rm Mod}$ via $H \mapsto \bigoplus_{\alpha \in I}H(V_{\alpha})$
\end{center}
is an equivalence with the inverse equivalence
\begin{center}
$\mathbb{F}: T_{\mathscr{A}}{\rm Mod} \rightarrow {\rm Mod}({\rm fp}(\mathscr{A}))$ via $M \mapsto \mathbb{F}_M$
\end{center}
where $\mathbb{F}_M: {\rm fp}(\mathscr{A}) \rightarrow \mathfrak{Ab}$ via $X \mapsto {\rm Hom}_{\mathscr{A}}(X, V) \otimes_{T_{\mathscr{A}}} M$. It is easy to see that it preserves and reflects finitely generated projective objects and coproducts. By using the composition of the functors $\mathbb{H}$ and $\mathbb{G}$, we get an additive equivalence $\mathscr{A} \rightarrow {\rm Flat}(T_{\mathscr{A}})$ via $A \mapsto \bigoplus_{\alpha \in I}{\rm Hom}_{\mathscr{A}}(V_{\alpha}, A)$ which induces an equivalence between ${\rm fp}(\mathscr{A})$ and ${\rm proj}(T_{\mathscr{A}})$. This is exactly the functor $$\widehat{{\rm Hom}}_{\mathscr{A}}(V, -): \mathscr{A} \rightarrow {\rm Flat}(T_{\mathscr{A}}).$$ Note that if every finitely presented object in a locally finitely presented Grothendieck category $\mathscr{A}$ is a finite coproduct of indecomposable objects in $\mathscr{A}$, then by \cite[Theorem 3.4]{zn}, $T_{\mathscr{A}}$ is Morita equivalent to the functor ring $R_{\mathscr{A}}$ of $\mathscr{A}$.

\begin{Pro}\label{r51}
Let $\mathscr{A}$ be a locally finitely presented pure semisimple Grothendieck category which satisfies the property $(\ast)$. Then $R_{\mathscr{A}}$ has the properties $(i)$-$(iv)$.
\end{Pro}
\begin{proof}
By \cite[Theorem 1.9]{S78}, every finitely presented object in $\mathscr{A}$ is a finite coproduct of indecomposable subobjects and ${\rm Mod}({\rm fp}(\mathscr{A}))$ is a perfect category and so $R_{\mathscr{A}}$ is a left perfect ring. Then there is a complete set $\lbrace \theta_i~|~i \in I \rbrace$ of pairwise orthogonal local idempotents of $R_{\mathscr{A}}$ such that $R_{\mathscr{A}}=\bigoplus_{i \in I}R_{\mathscr{A}}\theta_{i}=\bigoplus_{i \in I}\theta_{i}R_{\mathscr{A}}$. Since $\mathbb{H}: {\rm fp}(\mathscr{A}) \rightarrow {\rm proj}({\rm fp}(\mathscr{A}))$ is an additive equivalence,  $\mathbb{G}: {\rm Mod}({\rm fp}(\mathscr{A})) \rightarrow T_{\mathscr{A}}{\rm Mod}$ is an additive equivalence which preserves and reflects finitely generated projective objects and $R_{\mathscr{A}}$ is Morita equivalent to $T_{\mathscr{A}}$. By using \cite[Corollary 3]{gs1}, it is not difficult to show that each $\theta_{i}R_{\mathscr{A}}\theta_{i}$ is a left perfect ring and hence $R_{\mathscr{A}}$ has the property $(i)$. Since $\mathscr{A}$ is a locally finitely presented pure semisimple Grothendieck category, by \cite[Theorem 6.3]{S77}, $\mathscr{A}$ is locally noetherian and every object in $\mathscr{A}$ is a coproduct of indecomposable noetherian objects. This implies that by \cite[Corollary 2.6]{S80}, ${\rm Mod}({\rm fp}(\mathscr{A}))$ is locally noetherian. By \cite[Proposition 4.3]{bs}, $R_{\mathscr{A}}$ is left locally noetherian. On the other hand, since $\mathscr{A}$ is locally finitely presented Grothendieck category and $R_{\mathscr{A}}$ is a left perfect ring, by \cite[Theorems 2.7 and 3.3]{gs}, ${\rm gl.dim}R_{\mathscr{A}}= 0~ {\rm or} ~2$ and ${\rm dom.dim}R_{\mathscr{A}}\theta_i \geq 2$ for each $i \in I$. Therefore $R_{\mathscr{A}}$ has the properties $(ii)$ and $(iv)$. Now we show that $R_{\mathscr{A}}$ has the property $(iii)$. Since ${\rm Mod}({\rm fp}(\mathscr{A}))$ is a QF-3 category, $R_{\mathscr{A}}$ is a left QF-3 ring and so there are finitely many pairwise non-isomorphic simple left ideals $S_1, \cdots, S_n$ such that $M = \bigoplus_{l=1}^n E(S_l)$ is faithful, projective, and injective in $R_{\mathscr{A}}$Mod. We show that ${\rm proj}(R_{\mathscr{A}})\bigcap {\rm inj}(R_{\mathscr{A}})= {\rm add}(\bigoplus_{l=1}^n E(S_l))$. Assume that $N \in {\rm proj}(R_{\mathscr{A}})\bigcap {\rm inj}(R_{\mathscr{A}})$. Then $N \cong \bigoplus_{i=1}^tR_{\mathscr{A}}\theta_i$. Since $M$ is faithful, for each $i,~(1 \leq i \leq t)$, there is a monomorphism $R_{\mathscr{A}}\theta_i \rightarrow \widehat{\prod}_{A} M$, where $A$ is a set. But since every representable functor has finitely generated essential socle, for each $i,~(1 \leq i \leq t)$, there is a monomorphism $R_{\mathscr{A}}\theta_i \rightarrow \bigoplus_{l \in B_i} E(S_l)$, where $B_i$ is a finite set. Since $N$ is injective in $R_{\mathscr{A}}$Mod and each ${\rm End}_{R_{\mathscr{A}}}(E(S_l))$ is a local ring, $N \in {\rm add}(\bigoplus_{l=1}^n E(S_l))$ and so $R_{\mathscr{A}}$ has the property $(iii)$. Therefore the result follows.
\end{proof}

\begin{Pro}\label{r52}
Let $R$ be a ring with the properties $(i)$-$(iv)$. Then ${\rm Flat}(R)$ is a locally finitely presented pure semisimple Grothendieck category which satisfies the property $(\ast)$.
\end{Pro}
\begin{proof}
We know that ${\rm Flat}(R)$ is a locally finitely presented category and ${\rm fp}({\rm Flat}(R))$ coincides with the category of all finitely generated projective unitary left $R$-modules. By the proof of Theorem \ref{re1}, we can see that there exists a left pure semisimple ring $\Lambda$ such that $R_{\Lambda}$ is Morita equivalent to $R$. It implies that by the proof of Proposition \ref{l7}, $R$ is a left perfect right locally coherent QF-3 ring. Properties $(ii)$ and $(iv)$ and \cite[Theorem 2.7]{gs} implies that ${\rm Flat}(R)$ is a Grothendieck category. Therefore ${\rm Flat}(R)$ is a locally finitely presented Grothendieck category. Set $\mathscr{A}={\rm Flat}(R)$. Since ${\rm fp}(\mathscr{A})={\rm proj}(R)$, there is an additive equivalence ${\rm proj}(R) \rightarrow {\rm proj}(T_{\mathscr{A}})$. By \cite[Theorem 3.4]{zn}, $R$ is Morita equivalent to $T_{\mathscr{A}}$. But because $R$ is a left perfect left QF-3 ring and there is an additive equivalence ${\rm Mod}({\rm fp}(\mathscr{A})) \rightarrow T_{\mathscr{A}}{\rm Mod}$, ${\rm Mod}({\rm fp}(\mathscr{A}))$ is a  perfect QF-3 category. In particular, by \cite[Theorem 1.9]{S78}, $\mathscr{A}$ is a locally finitely presented pure semisimple Grothendieck category. Now we show that every representable functor in ${\rm Mod}({\rm fp}(\mathscr{A}))$ has finitely generated essential socle. Since $R$ is a left locally noetherian ring with the properties $(i)$ and $(iii)$, by using \cite[Ex 18, p. 134]{bs}, we can see that the injective hull of any finitely generated projective unitary left $R$-module is a finite direct sum of injective hull of simple modules. There is an additive equivalence ${\rm Mod}({\rm fp}(\mathscr{A})) \rightarrow R{\rm Mod}$ which preserves and reflect finitely generated projective objects and so by \cite[Ex 25(ii), p. 134]{bs}, every representable functor in ${\rm Mod}({\rm fp}(\mathscr{A}))$ has finitely generated essential socle. Therefore ${\rm Flat}(R)$ is a locally finitely presented pure semisimple Grothendieck category which satisfies the property $(\ast)$.
\end{proof}

We are now in a position to prove our main result in this section.

\begin{The}\label{r53}
There exists a bijection between Morita equivalence classes of rings with the properties $(i)$-$(iv)$ and equivalence classes of locally finitely presented pure semisimple Grothendieck categories which satisfy the property $(\ast)$.
\end{The}
\begin{proof}
Let $R$ be a ring with the properties $(i)$-$(iv)$. By Proposition \ref{r52}, ${\rm Flat}(R)$ is a locally finitely presented pure semisimple Grothendieck category which satisfies the property $(\ast)$. Let $R$ and $T$ be rings with the properties $(i)$-$(iv)$. If $R$ and $T$ are Morita equivalent, then there exists an additive equivalence ${\rm Flat}(R) \rightarrow {\rm Flat}(T)$. It implies that the map $R \mapsto {\rm Flat}(R)$ sends Morita equivalence classes of rings with the properties $(i)$-$(iv)$ to equivalence classes of locally finitely presented pure semisimple Grothendieck categories which satisfy the property $(\ast)$. We show that this mapping is a bijection. In order to show that this map is injective we suppose that $R$ and $T$ are rings with the properties $(i)$-$(iv)$ and there exists an additive equivalence ${\rm Flat}(R) \rightarrow {\rm Flat}(T)$. Then there exists an additive equivalence ${\rm proj}(R) \rightarrow {\rm proj}(T)$. By \cite[Theorem 3.4]{zn}, $R$ is Morita equivalent to $T$ and so this mapping is injective. It remains to show that this mapping is a surjective. Let $\mathscr{A}$ be a locally finitely presented pure semisimple Grothendieck category which satisfies the property $(\ast)$. Then by Proposition \ref{r51}, $R_{\mathscr{A}}$ is a ring with the properties $(i)$-$(iv)$ and also there is an additive equivalence $\mathscr{A} \rightarrow {\rm Flat}(R_{\mathscr{A}})$. Therefore the mapping is surjective and the result follows.
\end{proof}

As a consequence of Theorems \ref{re1} and \ref{r53}, we have the following result.

\begin{Cor}\label{r54}
There exists a bijection between Morita equivalence classes of left pure semisimple rings $\Lambda$ and equivalence classes of locally finitely presented pure semisimple Grothendieck categories which satisfy the property $(\ast)$.
\end{Cor}

We recall from \cite[p. 202]{dg} that a Grothendieck category $\mathscr{C}$ is called {\it locally finite} if it has a generating set of objects of finite length. A Grothendieck category $\mathscr{C}$ is called {\it of finite representation type} if it is locally finite and has only finitely many non-isomorphic finitely generated indecomposable objects (see \cite[p. 216]{dg}). Note that every locally finitely presented Grothendieck category of finite representation type is pure semisimple (see \cite[Proposition 3.3 and Lemma 4.1]{dg}).

\begin{Lem}\label{p1}
Let $\mathscr{A}$ be a locally finitely presented Grothendieck category of finite representation type. Then $\mathscr{A}$ has the property $(\ast)$.
\end{Lem}
\begin{proof}
By using \cite[Corollary 3]{gs1}, \cite[Theorem 2.7]{gs} and the existence of an additive equivalence $\mathscr{A} \rightarrow {\rm Flat}(R_{\mathscr{A}})$, the functor ring $R_{\mathscr{A}}$ of $\mathscr{A}$ is a left artinian ring with ${\rm dom.dim}R_{\mathscr{A}} \geq 2$ and ${\rm l.gl.dim}R_{\mathscr{A}}= 0 ~{\rm or}~ 2$ and so $R_{\mathscr{A}}$ is a QF-3 ring. Since $R_{\mathscr{A}}$ is left artinian, every finitely generated projective left $R_{\mathscr{A}}$-module has finitely generated essential socle. The existence of an additive equivalence ${\rm Mod}({\rm fp}(\mathscr{A})) \rightarrow R_{\mathscr{A}}$Mod which preserves and reflects finitely generated projective objects implies that ${\rm Mod}({\rm fp}(\mathscr{A}))$ is a QF-3 category and every representable functor has finitely generated essential socle.
\end{proof}

\begin{Lem}\label{p3}
Let $R$ be a left artinian ring with ${\rm l.gl.dim}R = 0~{\rm or}~ 2$ and ${\rm l.dom.dim}{R} \geq 2$. Then the category ${\rm Flat}(R)$ is a locally finitely presented Grothendieck category of finite representation type.
\end{Lem}
\begin{proof}
Since $R$ has the properties $(i)$-$(iv)$, by Proposition \ref{r52}, $\mathscr{A}={\rm Flat}(R)$ is a locally finitely presented pure semisimple Grothendieck category. It implies that by \cite[Theorem 6.3]{S77}, every indecomposable object in $\mathscr{A}$ is finitely generated. On the other hand, since $R$ is a left artinian ring, $\mathscr{A}$ has only finitely many non-isomorphic indecomposable objects. In particular, by \cite[Theorems 3.2 and 4.2]{dg}, $\mathscr{A}$ is locally finite. Therefore $\mathscr{A}$ is a locally finitely presented Grothendieck category of finite representation type.
\end{proof}

\begin{Cor}{\rm (See \cite[Theorem 4.4]{Ausla2})}\label{r56}
There exists a bijection between Morita equivalence classes of rings of finite representation type and equivalence classes of locally finitely presented Grothendieck categories of finite representation type.
\end{Cor}
\begin{proof}
It follows from Corollary \ref{re2}, Theorem \ref{r53},  Lemmas \ref{p1} and \ref{p3} and the proof of Lemma \ref{p1}.
\end{proof}

As a consequence of Corollaries \ref{r54} and \ref{r56}, we have the following result.

\begin{Cor}\label{p2}
The following statement are equivalent.
\begin{itemize}
\item[$(a)$] The pure semisimplicity conjecture is true;
\item[$(b)$] Every locally finitely presented pure semisimple Grothendieck category which satisfies the property $(\ast)$ is of finite representation type.
\end{itemize}
\end{Cor}

We recall that a quiver is a set of points connected together by some directed arrows (see \cite{gab1, gab2}). Let $K$ be a field and $Q$ be a quiver. A {\it representation $M$ of $Q$ over $K$} is given by attaching a $K$-vector space $M_v$ to each point $v$ of $Q$ and a $K$-linear map $M_u \rightarrow M_v$ to each arrow $\alpha: u \rightarrow v$ of $Q$. A morphism $\tau: M \rightarrow N$ between two representations consists of $K$-linear maps $\tau_v: M_v \rightarrow N_v$ for each point $v$ of $Q$ such that the obvious diagrams commute. We denote by ${\rm Rep}_K(Q)$ the category of all representations of $Q$ over $K$. Note that the category ${\rm Rep}_K(Q)$ is a Grothendieck category. \\

The category ${\rm Rep}_K(Q)$ of representations of a quiver $Q$ of pure semisimple type over a field $K$ which has been studied by Drozdowski and Simson \cite{ds} is one of the important classes of locally finitely presented pure semisimple Grothendieck categories. In the following example we show that this class of locally finitely presented pure semisimple Grothendieck categories satisfies the condition $(b)$ of Corollary \ref{p2}. \\

\begin{Examp}{\rm
Let $Q$ be a connected quiver and $K$ be a field. Then by \cite[Theorem 1]{ds},  $\mathscr{A}={\rm Rep}_K(Q)$ is a locally finitely presented pure semisimple Grothendieck category if and only if $Q$ is a subquiver of one of the following quivers:
\small{
\[ \xymatrix{^{\infty}{\mathbb{A}_n}^{\infty}:~
 \cdots   &\ar[l] {-1}  & 0 \ar[l]  \ar[r] &1 \ar@{-}[r] & \cdots  \ar@{-}[r]& n-1 &  n \ar[l] \ar[r] & n+1 \ar[r]& \cdots,~ n \geq 0, }\]}
 \small{
\[ \xymatrix{
& 1'' \ar@{-}[d] \\
{\mathbb{D}_n}^{\infty}:~ 1' \ar@{-}[r] & 0 \ar@{-}[r] & 1 \ar@{-}[r] & \cdots  \ar@{-}[r] & n-1 &  n \ar[l] \ar[r] & n+1 \ar[r]& \cdots, ~n \geq 0,}\]}
\small{
\[ \xymatrix{
& 1'' \ar@{-}[d] \\
\mathbb{E}_8:~ 2' \ar@{-}[r] & 1' \ar@{-}[r] & 0 \ar@{-}[r] & 1 \ar@{-}[r] & 2 \ar@{-}[r] & 3 \ar@{-}[r] & 4,}\]}
where \xymatrix{ i \ar@{-}[r] & j} means \xymatrix{ i \ar[r] & j} or \xymatrix{ i & \ar[l] j}. Let $Q$ be a subquiver of one of the quivers $^{\infty}{\mathbb{A}_n}^{\infty}$ and ${\mathbb{D}_n}^{\infty}$. Assume that ${\rm Mod}({\rm fp}(\mathscr{A}))$ is a QF-3 category. Then there exists a finitely generated projective injective faithful object in ${\rm Mod}({\rm fp}(\mathscr{A}))$. Let $M$ be a finitely generated projective injective faithful object in ${\rm Mod}({\rm fp}(\mathscr{A}))$. Then $M \cong (-, X)|_{{\rm fp}(\mathscr{A})}$ for some $X \in {\rm fp}(\mathscr{A})$. It is not difficult to see that $X$ is injective in $\mathscr{A}$ and for any non-zero morphism $f: A_1 \rightarrow A_2$, there exists a non-zero morphism $h: A_2 \rightarrow X$ such that $fh \neq 0$. By using \cite[Corollary 1 and Theorem 4]{ds}, we can see that the underlying graph of the quiver $Q$ is one of the Dynkin diagrams $\mathbb{A}_n$ or $\mathbb{D}_n$. Therefore the locally finitely presented pure semisimple Grothendieck category $\mathscr{A}$ satisfies the condition $(b)$ of Corollary \ref{p2}.}
\end{Examp}

In \cite{S80}, Simson studied purity for locally finitely presented hereditary Grothendieck categories. He showed that there exists a bijection between equivalence classes of locally finitely presented pure semisimple hereditary Grothendieck categories and equivalence classes of pure semisimple hereditary functor categories (see \cite[Corollary 2.9]{S80}). The proof of Theorems \ref{re1} and \ref{r53} shows that we have the following result.

\begin{Cor}\label{r64}
There exists a bijection between Morita equivalence classes of left pure semisimple left hereditary rings $\Lambda$ and equivalence classes of locally finitely presented pure semisimple hereditary Grothendieck categories which satisfy the property $(\ast)$.
\end{Cor}

Herzog in \cite{her} showed that, to prove the pure semisimplicity conjecture, it suffices to prove it for hereditary rings. As a consequence of Corollary \ref{p2} and \cite[Theorem 6.9]{her}, we have the following result.

\begin{Cor}\label{p22}
The following statement are equivalent.
\begin{itemize}
\item[$(a)$] The pure semisimplicity conjecture is true;
\item[$(b)$] Every locally finitely presented pure semisimple hereditary Grothendieck category which satisfies the property $(\ast)$ is of finite representation type.
\end{itemize}
\end{Cor}

\section*{acknowledgements}
The research of the first author was in part supported by a grant from IPM. Also, the research of the second author was in part supported by a grant from IPM (No. 14020416). The work of the second author is
based upon research funded by Iran National Science Foundation (INSF) under project No. 4001480.

\end{document}